\documentclass[a4paper,11pt]{amsart}
\usepackage[left=2.7cm,right=2.7cm,top=3.5cm,bottom=3cm]{geometry}

\usepackage{amsthm,amssymb,amsmath,amsfonts,mathrsfs,amscd}
\usepackage[all]{xy}
\usepackage{latexsym}
\usepackage{longtable}
\usepackage{dsfont}

\newfont{\cyr}{wncyr10 scaled 1100}
\newfont{\cyrr}{wncyr9 scaled 1000}

\theoremstyle{plain}
\newtheorem{theorem}{Theorem}[section]

\newtheorem{proposition}[theorem]{Proposition}
\newtheorem{lemma}[theorem]{Lemma}
\newtheorem{corollary}[theorem]{Corollary}

\newtheorem{introthm}{Theorem}
\newtheorem{introconj}[introthm]{Conjecture}
\newtheorem{introcor}[introthm]{Corollary}
\theoremstyle{remark}

\theoremstyle{definition}

\newtheorem{definition}[theorem]{Definition}

\newtheorem*{ass}{Assumption}
\newtheorem*{assSU}{Assumption}

\theoremstyle{remark}
\newtheorem{remark}[theorem]{Remark}

\newcommand{\sk}{\vspace{0.1in}}

\newcommand{\Q}{\mathbf Q}

\newcommand{\Z}{\mathbf Z}

\newcommand{\bQ}{\mathbf Q}

\newcommand{\bC}{\mathbf C}

\DeclareMathOperator{\Hom}{Hom}

\DeclareMathOperator{\Gal}{Gal}
\DeclareMathOperator{\GL}{\rm GL}

\DeclareMathOperator{\M}{M}

\newcommand{\val}{\mathrm{val}}

\newcommand{\ord}{\mathrm{ord}}
\newcommand{\new}{\mathrm{new}}

\usepackage[usenames]{color}
\definecolor{Indigo}{rgb}{0.2,0.1,0.7}
\definecolor{Violet}{rgb}{0.5,0.1,0.7}
\definecolor{White}{rgb}{1,1,1}
\definecolor{Green}{rgb}{0.1,0.9,0.2}

\newcommand{\mat}[4]{\left(\begin{array}{cc}#1&#2\\#3&#4\end{array}\right)}
\newcommand{\smallmat}[4]{\bigl(\begin{smallmatrix}#1&#2\\#3&#4\end{smallmatrix}\bigr)}
\newcommand{\pwseries}[1]{[[#1]]}

\newcommand{\cO}{{\mathcal O}}

\newcommand{\F}{\mathbf f}

\def\SU{(SU)}

\include{thebibliography}

\begin{document}

\title[Anticyclotomic Iwasawa invariants in Hida families]{Variation of anticyclotomic Iwasawa invariants\\ in Hida families}
\author[F.~Castella]{Francesc Castella}
\author[C.-H.~Kim]{Chan-Ho Kim}
\author[M.~Longo]{Matteo Longo}
\address{Department of Mathematics, Princeton University, Fine Hall, Princeton, NJ 08544, USA}
\email{fcabello@math.princeton.edu}
\address{School of Mathematics, Korea Institute for Advanced Study (KIAS), 85 Hoegiro, Dongdaemun-Gu, Seoul 02455, Republic of Korea}
\email{chanho.math@gmail.com}
\address{Dipartimento di Matematica, Universit\`a di Padova, Via Trieste 63, 35121 Padova, Italy}
\email{mlongo@math.unipd.it}

\maketitle

\begin{abstract}
Building on the construction of big Heegner points in the quaternionic setting \cite{LV-MM},
and their relation to special values of Rankin--Selberg $L$-functions \cite{cas-longo}, we obtain 
anticyclotomic analogues of the results of Emerton--Pollack--Weston \cite{EPW} on the variation of Iwasawa
invariants in Hida families. In particular, combined with the known cases of the 
anticyclotomic Iwasawa main conjecture in weight $2$, our results yield a proof of the main conjecture
for $p$-ordinary newforms of higher weights and trivial nebentypus.
\end{abstract}

\maketitle
\tableofcontents

\section*{Introduction}

In a remarkable paper \cite{EPW}, Emerton--Pollack--Weston obtained striking results on
the behaviour of the cyclotomic Iwasawa invariants attached to $p$-ordinary modular forms
as they vary in Hida families. In particular, combined with Greenberg's conjecture on the vanishing of
the $\mu$-invariant, 
their main result 
\emph{reduces} the proof of the main conjecture to the weight two case.
In this paper, we develop analogous results for newforms based-changed
to imaginary quadratic fields in the definite anticyclotomic setting.
In particular, combined with Vatsal's result on the vanishing of the anticyclotomic $\mu$-invariant \cite{Vat1},
and the known cases of the anticyclotomic main conjecture in weight $2$
(thanks to the works of Bertolini--Darmon \cite{bdIMC},
Pollack--Weston \cite{pollack-weston}, and Skinner--Urban \cite{SU}),
our results yield a proof of Iwasawa's main conjecture for $p$-ordinary modular forms
of higher weights $k\geqslant 2$
and trivial nebentypus in the anticyclotomic setting.
\sk

Let us begin by recalling the setup of \cite{EPW}, but adapted to the
context at hand. Let
\[
\bar\rho:G_\Q:={\rm Gal}(\overline{\Q}/\Q)\longrightarrow{\rm GL}_2(\mathbf{F})
\]
be a continuous Galois representation defined over a finite field $\mathbf{F}$ of characteristic $p>3$,
and assume that $\bar{\rho}$ is odd and irreducible.
After the proof of Serre's conjecture \cite{kw}, we know that $\bar\rho$ is modular,
meaning that $\bar\rho$ is isomorphic to the mod $p$ Galois representation $\bar{\rho}_{f_0}$ 
associated to an elliptic newform $f_0$. 
Throughout this paper, it will be assumed that $\bar\rho\simeq\bar{\rho}_{f_0}$ for some newform $f_0$ of weight $2$ and trivial nebentypus.
\sk

Let $N(\bar\rho)$ be the tame conductor of $\bar\rho$, and let $K/\Q$ be an imaginary quadratic field
of discriminant prime $-D_K<0$ to $pN(\bar\rho)$. The field $K$ then determines a decomposition
\[
N(\bar{\rho})=N(\bar\rho)^+\cdot N(\bar\rho)^-
\]
with $N(\bar\rho)^+$ (resp. $N(\bar\rho)^-$)
only divisible by primes which are split (resp. inert) in $K$. We similarly define
the decomposition $M=M^+\cdot M^-$ for any positive integer $M$ prime to $D_K$.
\sk

As in \cite{pollack-weston},
we consider the following conditions on a pair $(\bar\rho,N^-)$, where $N^-$ is a fixed
square-free product of an odd number of primes inert in $K$:

\begin{ass}[CR]\label{CR}\hfill
\begin{enumerate}
\item{} $\bar{\rho}$ is irreducible; 
\item{} $N(\bar\rho)^-\mid N^-$;
\item{} $\bar{\rho}$ is ramified at every prime $\ell\mid N^-$ such that $\ell\equiv\pm{1}\pmod{p}$.
\end{enumerate}
\end{ass}

Let $\mathcal{H}(\bar{\rho})$ be the set of all $p$-ordinary and $p$-stabilized
newforms with mod $p$ Galois representation isomorphic to $\bar\rho$, and let
$\Gamma:={\rm Gal}(K_\infty/K)$ denote the Galois group of the anticyclotomic $\Z_p$-extension of $K$.
Associated with each $f\in\mathcal{H}(\bar\rho)$ of tame level $N_f$ with $N_f^-=N^-$,
defined over say a finite extension $F/\Q_p$ with ring of integers $\cO$, there is a $p$-adic $L$-function
\[
L_p(f/K)\in\cO\pwseries{\Gamma}
\]
constructed by Bertolini--Darmon \cite{BDmumford-tate} in weight two, and
by Chida--Hsieh \cite{ChHs1} for higher weights. The $p$-adic $L$-function $L_p(f/K)$ is characterized,
as $\chi$ runs over the $p$-adic characters of $\Gamma$ corresponding to certain algebraic Hecke characters of $K$,
by an interpolation property of the form
\begin{equation}\label{eq:interp}
\chi(L_p(f/K))=C_p(f,\chi)
\cdot E_p(f,\chi)\cdot\frac{L(f,\chi,k/2)}{\Omega_{f,N^-}},\nonumber
\end{equation}
where $C_p(f,\chi)$ is an explicit nonzero constant, $E_p(f,\chi)$ is a $p$-adic multiplier,
and $\Omega_{f,N^-}$ is a complex period 
making the above ratio algebraic. (Of course, implicit in all the above is a fixed
choice of complex and $p$-adic embeddings $\mathbf{C}\overset{\iota_\infty}
\hookleftarrow\overline\Q\overset{\iota_p}\hookrightarrow\overline\Q_p$.)
\sk

The anticyclotomic Iwasawa main conjecture gives an arithmetic interpretation of $L_p(f/K)$.
More precisely, let
\[
\rho_f:G_{\bQ}\longrightarrow{\rm Aut}_F(V_f)\simeq{\rm GL}_2(F)
\]
be a self-dual twist of the $p$-adic Galois representation associated to $f$, fix an $\cO$-stable lattice $T_f\subseteq V_f$,
and set $A_f:=V_f/T_f$. Let $D_p\subseteq G_\bQ$ be the decomposition group corresponding to our fixed
embedding $\iota_p$, and let $\varepsilon_{\rm cyc}$ be the $p$-adic
cyclotomic character. Since $f$ is $p$-ordinary, there is a unique
one-dimensional $D_p$-invariant subspace
$F_p^+V_f\subseteq V_f$ where the inertia group at $p$ acts via $\varepsilon_{\rm cyc}^{k/2}\psi$,
with $\psi$ a finite order character.
Let $F_p^+A_f$ be the image of $F_p^+V_f$ in $A_f$ and set $F_p^-A_f:=A_f/F_p^+A_f$.
Following the terminology in \cite{pollack-weston}, the \emph{minimal Selmer group} of $f$ is defined by
\[
{\rm Sel}(K_\infty,f):=\ker\left\{H^1(K_\infty,A_f)
\longrightarrow
\prod_{w\nmid p}H^1(K_{\infty,w},A_f)\times\prod_{w\vert p}H^1(K_{\infty,w},F_p^-A_f)\right\},
\]
where $w$ runs over the places of $K_\infty$. By standard arguments (see \cite{greenberg-Iw}, for example),
one knows that the Pontryagin dual of ${\rm Sel}(K_\infty,f)$ is finitely generated
over the anticyclotomic Iwasawa algebra $\Lambda:=\cO[[\Gamma]]$. 
The \emph{anticyclotomic main conjecture} is then the following:

\begin{introconj}
\label{AIMC}
The Pontryagin dual ${\rm Sel}(K_\infty,f)^\vee$ is $\Lambda_{}$-torsion, and
\[
Ch_{\Lambda_{}}({\rm Sel}(K_\infty,f)^\vee)=(L_p(f/K)).
\]
\end{introconj}

For newforms $f$ of weight $2$ corresponding to elliptic curves $E/\bQ$ with ordinary reduction at $p$,
and under rather stringent assumptions on $\bar\rho_f$ which were later relaxed
by Pollack--Weston \cite{pollack-weston}, one of the divisibilities predicted by Conjecture~\ref{AIMC}
was obtained by Bertolini--Darmon \cite{bdIMC} using Heegner points and
Kolyvagin's method of Euler systems. 
More recently, after the work of Chida--Hsieh \cite{ChHs2}
the divisibility
\begin{equation}\label{ES-div}
Ch_{\Lambda_{}}({\rm Sel}(K_\infty,f)^\vee)\supseteq(L_p(f/K))\nonumber
\end{equation}
is known for newforms $f$ of weight $k\leqslant p-2$ and trivial nebentypus,
provided the pair $(\bar{\rho}_f,N_f^-)$ satisfies a mild strengthening of Hypotheses~(CR). 
This restriction to small weights comes from the use of
Ihara's lemma \cite{DT-inv}, and it seems difficult to directly extend their arguments
in \cite{ChHs2} to higher weights. Instead, as we shall explain in the following paragraphs,
in this paper we will complete the proof of Conjecture~\ref{AIMC} to all weights $k\equiv 2\pmod{p-1}$ 
by a different approach, using Howard's big Heegner points in Hida families \cite{howard-invmath}, as
extended by Longo--Vigni \cite{LV-MM} to quaternionic Shimura curves.
\sk

Associated with every $f\in\mathcal{H}(\bar\rho)$ there are \emph{anticyclotomic Iwasawa invariants}
$\mu^{\rm an}(K_\infty,f)$, $\lambda^{\rm an}(K_\infty,f)$, $\mu^{\rm alg}(K_\infty,f)$, and $\lambda^{\rm alg}(K_\infty,f)$.
The analytic (resp. algebraic) $\lambda$-invariants are the
number of zeros of $L_p(f/K)$ (resp. of a generator of the characteristic ideal of ${\rm Sel}(K_\infty,f)^\vee$), while
the $\mu$-invariants are defined as the exponent of the highest power of $\varpi$ (with $\varpi\in\cO$ any uniformizer)
dividing the same objects.
Our main results on the variation of these invariants are summarized in the following. (Recall that we assume $\bar\rho\simeq\bar{\rho}_{f_0}$ for some newform $f_0$ of weight $2$ and trivial nebentypus.)

\begin{introthm}\label{thmA}
Assume in addition that: 
\begin{itemize}
\item{} $\bar\rho$ is irreducible;
\item{} $\bar\rho$ is $p$-ordinary, ``non-anomalous" and $p$-distinguished:
\[
\bar{\rho}\vert_{D_p}\simeq\left(\begin{smallmatrix}\bar{\varepsilon}&*\\0&\bar{\delta}\end{smallmatrix}\right)
\]
with $\bar{\varepsilon}, \bar{\delta}:D_p\rightarrow\mathbf{F}^\times$ characters such that $\bar{\delta}$
is unramified, $\bar{\delta}({\rm Frob}_p)\neq\pm 1$
and $\bar\delta\neq\bar{\varepsilon}$;
\item{} $N(\bar\rho)^-$ is the square-free product of an odd number of primes.
\end{itemize}
Let $\mathcal{H}^-(\bar\rho) := \mathcal{H}^{N(\bar\rho)^-}(\bar{\rho})$ consist of
all newforms $f\in\mathcal{H}(\bar{\rho})$ with $N_f^-=N(\bar\rho)^-$, and fix $*\in\{{\rm alg}, {\rm an}\}$. Then the following hold:
\begin{enumerate}
\item For all $f\in\mathcal{H}^-(\bar\rho)$, we have
\[
\mu^*(K_\infty,f)=0.
\]
\item Let $f_1, f_2\in\mathcal{H}^-(\bar\rho)$ lie on the branches
$\mathbb{T}(\mathfrak{a}_1)$, $\mathbb{T}(\mathfrak{a}_2)$ (defined in $\S$\ref{sec:branches}), respectively. Then
\begin{equation}\label{lambda}
\lambda^*(K_\infty,f_1)-\lambda^*(K_\infty,f_2)=\sum_{\ell\mid N^+_{f_1}N^+_{f_2}}
e_\ell(\mathfrak{a}_2)-e_\ell(\mathfrak{a}_1),\nonumber
\end{equation}
where the sum is over the split primes in $K$ which divide the tame level of $f_1$ or $f_2$,
and $e_\ell(\mathfrak{a}_j)$ is an explicit non-negative invariant of the branch $\mathbb{T}(\mathfrak{a}_j)$
and the prime $\ell$.
\end{enumerate}
\end{introthm}
Provided that $p$ splits in $K$, and under the same hypotheses on $\bar\rho$ as in Theorem~\ref{thmA},
the work of Skinner--Urban \cite{SU} establishes one of the divisibilities in a related
``three-variable'' Iwasawa main conjecture. Combining their work with our Theorem~\ref{thmA},
and making use of the aforementioned results of Bertolini--Darmon \cite{bdIMC} and
Pollack--Weston \cite{pollack-weston} in weight $2$, we obtain many new cases of Conjecture~\ref{AIMC}
(\emph{cf.} Corollary \ref{coro:5.4}):

\begin{introcor}\label{corA}
Suppose that $\bar\rho$ is as in Theorem$~\ref{thmA}$ and that $p$ splits in $K$.
Then the anticyclotomic Iwasawa main conjecture holds for every $f\in\mathcal{H}^-(\bar\rho)$
of weight $k\equiv 2\pmod{p-1}$ and trivial nebentypus.
\end{introcor}

Let us briefly explain the new ingredients in the proof of Theorem~\ref{thmA}.
As it will be clear to the reader, the results contained in Theorem~\ref{thmA}
are anticyclotomic analogues of the results
of Emerton--Pollack--Weston \cite{EPW} in the cyclotomic setting. In fact, on the algebraic side the arguments of
\emph{loc.cit.} carry over almost verbatim,
and our main innovations in this paper are in the development of anticyclotomic analogous of their results on the \emph{analytic} side.
Indeed, the analytic results of \cite{EPW} are based on the study of certain two-variable
$p$-adic $L$-functions \`a la Mazur--Kitagawa, whose construction relies on the theory of modular symbols
on classical modular curves. In contrast, 
we need to work on a family of Shimura curves associated with definite quaternion algebras,
for which cusps are not available. In the cyclotomic case, modular symbols are useful two ways: They
yield a concrete realization of the degree one compactly supported cohomology of open modular curves,
and provide a powerful tool for studying the arithmetic properties of critical values
of the $L$-functions attached to modular forms. Our basic observation is that in the present anticyclotomic setting,
Heegner points on definite Shimura curves 
provide a similarly convenient way of describing the \emph{central} critical values of the
Rankin $L$-series $L(f/K,\chi,s)$.
\sk

Also fundamental for the method of \cite{EPW} is the possibility to ``deform''
modular symbols in Hida families. In our anticyclotomic context, the construction of big
Heegner points in Hida families was obtained in the work  \cite{LV-MM} of the third-named author
in collaboration with Vigni, 
while the relation between these points and Rankin--Selberg $L$-values
was established in the work \cite{cas-longo} by the first- and third-named authors.
With these key results at hand, and working over appropriate quotients of the Hecke algebras considered in
\cite{EPW} via the Jaquet--Langlands correspondence, we are then able to develop analogues of the arguments of \emph{loc.cit.}
in our setting, making use of the ramification hypotheses on $\bar\rho$
to ensure a multiplicity one property of certain Hecke modules, 
similarly as in the works of Pollack--Weston \cite{pollack-weston} and the second-named author \cite{kim-summary}.
\sk

We conclude this Introduction with an overview of the contents of the paper.
In the next section, we briefly recall the Hida theory that we need, following the exposition in \cite[\S{1}]{EPW} for the most part.
In Section~\ref{sec:heegner}, we describe a key extension of the construction
of big Heegner points of \cite{LV-MM} to ``imprimitive'' branches of the Hida family.
In Section~\ref{sec:p-adicL}, we construct two-variable
$p$-adic $L$-functions attached to a Hida family and to each of its irreducible components (or branches), and prove Theorem~\ref{thm:3.6.2}
relating the two. This theorem is the key technical result of this paper, and the analytic part of Theorem~\ref{thmA}
follows easily from this. In Section~\ref{sec:Selmer}, we deduce the algebraic part of Theorem~\ref{thmA}
using the residual Selmer groups studied in \cite[\S{3.2}]{pollack-weston}. Finally, in Section~\ref{sec:applications}
we give the applications of our results to the anticyclotomic Iwasawa main conjecture.
\sk

\emph{Acknowledgements.} During the preparation of this paper, F.C. was partially supported 
by the European Research Council (ERC) under the European Union's
Horizon 2020 research and innovation programme (grant agreement No. 682152); C.K. was partially supported by an AMS--Simons Travel Grant;
M.L. was partially supported by PRIN 2010-11 ``Arithmetic Algebraic Geometry and Number Theory'' and
by PRAT 2013 ``Arithmetic of Varieties over Number Fields''.

\section{Hida theory}\label{sec:hida-theory}

Throughout this chapter, we fix a positive integer $N$ admitting a factorization
\[
N=N^+ N^-
\] 
with $(N^+,N^-)=1$ and $N^-$ equal to the square-free product of an \emph{odd} number of primes.
We also fix a prime $p\nmid 6N$.

\subsection{Hecke algebras}\label{subsec:hecke}

For each integer $k\geqslant 2$, denote by $\mathfrak{h}_{N,r,k}$ the $\Z_p$-algebra
generated by the Hecke operators $T_\ell$ for $\ell\nmid Np$, the operators $U_\ell$ for $\ell\mid Np$,
and the diamond operators $\langle a\rangle$ for $a\in(\Z/p^r\Z)^\times$,
acting on the space $S_k(\Gamma_{0,1}(N,p^r),\overline{\Q}_p)$ of cusp forms of weight $k$ on
$\Gamma_{0,1}(N,p^r):=\Gamma_0(N)\cap\Gamma_1(p^r)$.
For $k=2$, we abbreviate $\mathfrak h_{N,r}:=\mathfrak h_{N,r,2}$.
\sk

Let $e^{\rm ord}:=\lim_{n\to\infty}U_p^{n!}$ be Hida's ordinary projector, and define
\[
\mathfrak{h}_{N,r,k}^{\rm ord}:=e^{\rm ord}\mathfrak{h}_{N,r,k},\qquad
\mathfrak{h}_{N,r}^{\rm ord}:=e^{\rm ord}\mathfrak{h}_{N,r},\qquad
\mathfrak{h}_{N}^{\rm ord}:=\varprojlim_r\mathfrak{h}^{\rm ord}_{N,r},
\]
where the limit is over the projections induced by the natural restriction maps.
\sk

Let $\mathbb{T}^{N^-}_{N,r,k}$ be the quotient of $\mathfrak{h}_{N,r,k}^\ord$
acting faithfully on the subspace of $e^{\rm ord}S_k(\Gamma_{0,1}(N,p^r),\overline{\Q}_p)$ consisting of forms
which are new at all primes dividing $N^-$. Set $\mathbb T^{N^-}_{N,r}:=\mathbb T^{N^-}_{N,r,2}$ and define
\[
\mathbb{T}_{N}^{N^-}:=\varprojlim_r\mathbb{T}^{N^-}_{N,r}.
\]

Each of these Hecke algebras are equipped with natural $\Z_p\pwseries{\Z_p^\times}$-algebra structures via the diamond operators,
and by a well-known result of Hida, $\mathfrak{h}_N^{\rm ord}$ is finite and flat over $\Z_p\pwseries{1+p\Z_p}$.

\subsection{Galois representations on Hecke algebras}

For each positive integer $M\mid N$ we may consider
the new quotient $\mathbb T_M^\new$ of $\mathfrak h_M^\ord$, and the Galois representation
\begin{equation}\label{2.2.1}
\rho_M:G_\Q\longrightarrow\GL_2(\mathbb T_M^\new\otimes_{}\mathcal L)\nonumber
\end{equation}
described in \cite[Thm.~2.2.1]{EPW}, where $\mathcal L$ denotes the fraction field of $\Z_p\pwseries{1+p\Z_p}$.
\sk

Let $\mathbb T_N'$ be the $\Z_p\pwseries{1+p\Z_p}$-subalgebra of $\mathbb T_N^{N^-}$
generated by the image under the natural projection $\mathfrak{h}_{N}^{\rm ord}\rightarrow\mathbb{T}_N^{N^-}$
of the Hecke operators of level prime to $N$. As in \cite[Prop.~2.3.2]{EPW}, one can show that
the canonical map
\[
\mathbb T_N'\longrightarrow\prod_M\mathbb T_M^\new,
\]
where the product is over all integers $M\geqslant 1$ with $N^-\mid M\mid N$,
becomes an isomorphism after tensoring with $\mathcal L$.
Taking the product of the Galois representations $\rho_M$ we thus obtain 
\[
\rho:G_\Q\longrightarrow\GL_2(\mathbb T_N'\otimes\mathcal L).
\]

For any maximal ideal $\mathfrak m$ of $\mathbb T_N'$, let $(\mathbb T_N')_\mathfrak m$ denote the localization
of $\mathbb T_N'$ at $\mathfrak m$ and let
\[
\rho_\mathfrak m:G_\Q\longrightarrow \GL_2\left((\mathbb T_N ')_\mathfrak m\otimes\mathcal L\right)
\]
be the resulting Galois representation. If the residual representation
$\bar\rho_\mathfrak m$ 
is irreducible, then $\rho_{\mathfrak m}$ admits an integral model (still denoted in the same manner)
\[
\rho_\mathfrak m:G_\Q\longrightarrow\GL_2\left((\mathbb T_N')_\mathfrak m\right)
\]
which is unique up to isomorphism.

\subsection{Residual representations} \label{subsec:residual}

Let $\bar\rho:G_\Q\rightarrow\GL_2(\mathbf{F})$ be
an odd irreducible Galois representation defined over a finite field $\mathbf{F}$ of characteristic $p>3$.
As in the Introduction, we assume that $\bar\rho\simeq\bar{\rho}_{f_0}$ for some newform $f_0$ of weight $2$,
level $N$, and trivial nebentypus. Consider the following three conditions we may impose
on the pair $(\bar\rho,N^-)$:

\begin{assSU}[SU]\hfill
\begin{enumerate}
\item $\bar\rho$ is \emph{$p$-ordinary}: the restriction of $\bar\rho$ to a decomposition group $D_p\subseteq G_{\Q}$
at $p$ has a one-dimensional unramified quotient over $\mathbf{F}$;
\item $\bar\rho$ is \emph{$p$-distinguished}:
$\bar\rho_{}\vert_{D_p}\sim\left(\begin{smallmatrix}\bar{\varepsilon}&*\\0&\bar{\delta}\end{smallmatrix}\right)$
with $\bar{\varepsilon}\neq\bar{\delta}$;
\item $\bar\rho$ is ramified at every prime $\ell\mid N^-$.
\end{enumerate}
\end{assSU}

Fix once and for all a representation $\bar\rho$ satisfying Assumption~(SU),
together with a $p$-stabilization of $\bar\rho$ in the sense of \cite[Def.~2.2.10]{EPW}.
Let $\overline{V}$ be the two-dimensional $\mathbf{F}$-vector space which affords $\bar\rho$, and
for any finite set of primes $\Sigma$ that does not contain $p$ or any factor of $N^-$, define
\begin{equation}\label{Def:N(Sigma)}
N(\Sigma):=N(\bar\rho)\prod_{\ell\in\Sigma}\ell^{m_\ell},
\end{equation}
where $N(\bar{\rho})$ is the tame conductor of $\bar{\rho}$, and $m_\ell:={\rm dim}_{\mathbf{F}}\;\overline{V}_{I_\ell}$.

\begin{remark}
By Assumption~(SU) we have the divisibility $N^-\mid N(\bar\rho)$; we will further assume
that $(N^-,N(\bar\rho)/N^-)=1$.
\end{remark}

Combining \cite[Thm.~2.4.1]{EPW} and \cite[Prop.~2.4.2]{EPW} with the fact that $\bar\rho$ is
ramified at the primes dividing $N^-$, one can see that there exist unique maximal ideals
$\mathfrak n$ and $\mathfrak{m}$ of $\mathbb T_{N(\Sigma)}^{N^-}$ and $\mathbb T_{N(\Sigma)}'$, respectively,
such that
\begin{itemize}
\item 
$\mathfrak{n}\cap\mathbb T_{N(\Sigma)}'=\mathfrak{m}$;
\item $(\mathbb T_{N(\Sigma)}')_\mathfrak m\simeq(\mathbb T^{N^-}_{N(\Sigma)})_\mathfrak n$ by the natural map on localizations;
\item $\bar\rho_\mathfrak m\simeq\bar\rho$.
\end{itemize}
Define the ordinary Hecke algebra $\mathbb{T}_\Sigma$ attached to $\bar\rho$ and $\Sigma$ by
\[
\mathbb T_\Sigma:=(\mathbb T_{N(\Sigma)}')_{\mathfrak{m}}.
\]
Thus $\mathbb{T}_\Sigma$ is a local factor of $\mathbb{T}_{N(\Sigma)}'$, and we let
\[
\rho_\Sigma:G_\Q\longrightarrow \GL_2\left(\mathbb T_\Sigma\right)
\]
denote the Galois representation deduced from $\rho_\mathfrak{m}$.
\sk

Adopting the terminology of \cite[\S{2.4}]{EPW},
we shall refer to ${\rm Spec}(\mathbb{T}_\Sigma)$ as ``the Hida family'' $\mathcal{H}^-(\bar\rho)$ attached to $\bar\rho$
(and our chosen $p$-stabilization) that is minimally ramified outside $\Sigma$.

\begin{remark}
Note that by Assumption~\SU, all the $p$-stabilized newforms in $\mathcal{H}^-(\bar\rho)$
have tame level divisible by $N^-$.
\end{remark}

\subsection{Branches of the Hida family}\label{sec:branches}

If $\mathfrak a$ is a minimal prime of $\mathbb T_\Sigma$ (for a finite set of primes $\Sigma$ as above), we put
$\mathbb T (\mathfrak a):=\mathbb T_\Sigma /\mathfrak a$ and let
\[
\rho(\mathfrak a):G_\Q\longrightarrow\GL_2(\mathbb T(\mathfrak a))
\]
be the Galois representation induced by $\rho_\Sigma$. As in \cite[Prop.~2.5.2]{EPW}, one can show
that there is a unique divisor $N(\mathfrak a)$ of $N(\Sigma)$ and a unique minimal prime
$\mathfrak a'\subseteq\mathbb T_{N(\mathfrak a)}^\new$ above $\mathfrak{a}$ such that the diagram
\[
\xymatrix{
\mathbb T_\Sigma\ar[r] \ar[d]& \mathbb T'_{N(\Sigma)} \ar[r] & \prod_{N^-\mid M\mid N(\Sigma)}\mathbb T^\new_{M}\ar[d]\\
\mathbb{T}_{\Sigma}/\mathfrak{a} \ar[r]^-{=} & \mathbb T(\mathfrak a)\ar[r] & \mathbb T^\new_{N(\mathfrak a)}/\mathfrak a'}
\]
commutes. We call $N(\mathfrak a)$ the \emph{tame conductor} of $\mathfrak a$ and set
\[
\mathbb T(\mathfrak a)^\circ:=\mathbb T^\new_{N(\mathfrak a)}/\mathfrak a'.
\]

In particular, note that $N^-\mid N(\mathfrak a)$ by construction, and that the natural map
$\mathbb{T}(\mathfrak{a})\rightarrow\mathbb{T}(\mathfrak{a})^\circ$ is an embedding of local domains.

\subsection{Arithmetic specializations}

For any finite $\Z_p\pwseries{1+p\Z_p}$-algebra $\mathbb{T}$, we say that
a height one prime $\wp$ of $\mathbb{T}$ is an
\emph{arithmetic prime} of $\mathbb{T}$ if $\wp$ is the kernel of a $\Z_p$-algebra homomorphism
$\mathbb{T}\rightarrow\overline{\bQ}_p$ such that the composite map
\[
1+p\Z_p\longrightarrow\Z_p\pwseries{1+p\Z_p}^\times\longrightarrow\mathbb{T}^\times
\longrightarrow\overline{\Q}_p^\times
\]
is given by $\gamma\mapsto\gamma^{k-2}$ on some open subgroup of $1+p\Z_p$,
for some integer $k\geqslant 2$. We then say that $\wp$ has \emph{weight} $k$.
\sk

Let $\mathfrak a\subseteq\mathbb{T}_\Sigma$ be a minimal prime as above.
For each $n\geqslant 1$, let $\mathbf{a}_n\in\mathbb T(\mathfrak a)^\circ$ be the image of $T_n$
under the natural projection $\mathfrak{h}^{\rm ord}_{N(\Sigma)}\rightarrow\mathbb T(\mathfrak a)^\circ$, and form the $q$-expansion
\[
\F(\mathfrak a)=\sum_{n\geqslant 1}\mathbf{a}_nq^n\in\mathbb T(\mathfrak a)^\circ\pwseries{q}.
\]

By \cite[Thm.~1.2]{hida86b}, if $\wp$ is an arithmetic prime
of $\mathbb T(\mathfrak a)$ of weight $k$, then there is a unique height one prime
$\wp'$ of $\mathbb{T}(\mathfrak{a})^\circ$ such that
\[
\mathbf{f}_\wp(\mathfrak a):=\sum_{n\geqslant1}(\mathbf{a}_n\;{\rm mod}\;\wp')q^n\in\cO_\wp^\circ\pwseries{q},
\]
where $\cO_\wp^\circ:=\mathbb{T}(\mathfrak{a})^\circ/\wp'$,
is the $q$-expansion a $p$-ordinary eigenform $f_\wp$ of weight $k$ and tame level
$N(\mathfrak{a})$ (see \cite[Prop.~2.5.6]{EPW}).

\section{Big Heegner points}\label{sec:heegner}

As in Chapter~\ref{sec:hida-theory}, we fix an integer $N\geqslant 1$ admitting a factorization
$N=N^+N^-$ with $(N^+,N^-)=1$ and $N^-$ equal to the square-free product of an \emph{odd} number of primes, and
fix a prime $p\nmid 6N$. Also, we let $K/\Q$ be an imaginary quadratic field of discriminant $-D_K<0$ prime to $Np$
and such that every prime factor of $N^+$ (resp. $N^-$) splits (resp. is inert) in $K$.
\sk

In this section we describe a mild extension of the construction in \cite{LV-MM} (following \cite{howard-invmath})
of big Heegner points attached to $K$. Indeed, using the results from the preceding section,
we can extend the constructions of \emph{loc.cit.} to branches of the Hida family which are \emph{not necessarily}
primitive (in the sense of \cite[\S{1}]{hida86b}). The availability of such extension is fundamental for the
purposes of this paper.

\subsection{Definite Shimura curves}\label{subsec:Sh}

Let $B$ be the definite quaternion algebra
over $\Q$ of discriminant $N^-$. We fix once and for all an embedding of $\Q$-algebras
$K\hookrightarrow B$, and use it to identity $K$ with a subalgebra of $B$. Denote by $z\mapsto\overline{z}$
the nontrivial automorphism of $K$, and choose a basis $\{1,j\}$ of $B$ over $K$ such that
\begin{itemize}
\item $j^2=\beta\in\Q^\times$ with $\beta<0$;
\item $jt=\bar tj$ for all $t\in K$;
\item $\beta\in (\Z_q^\times)^2$ for $q\mid pN^+$, and $\beta\in\Z_q^\times$ for $q\mid D_K$.
\end{itemize}

Fix 
a square-root $\delta_K=\sqrt{-D_K}$, and
define $\boldsymbol{\theta}\in K$ by
\[
\boldsymbol{\theta}:=\frac{D'+\delta_K}{2},\quad\textrm{where}\quad
D':=\left\{
\begin{array}{ll}
D_K &\textrm{if $2\nmid D_K$;}\\
D_K/2 &\textrm{if $2\mid D_K$.}
\end{array}
\right.
\]
Note that $\cO_K=\Z+\Z\boldsymbol{\theta}$, and for every prime $q\mid pN^+$,
define $i_q:B_q:=B\otimes_\Q\Q_q \simeq \M_2(\Q_q)$ by
\[
i_q(\boldsymbol{\theta})=\mat{\mathrm{Tr}(\boldsymbol{\theta})}{-\mathrm{Nm}(\boldsymbol{\theta})}10,
\quad\quad
i_q(j)=\sqrt\beta\mat{-1}{\mathrm{Tr}(\boldsymbol{\theta})}01,
\]
where $\mathrm{Tr}$ and $\mathrm{Nm}$ are the reduced trace and reduced norm maps on $B$, respectively.
On the other hand, for each prime $q\nmid Np$ we fix any isomorphism $i_q:B_q\simeq \M_2(\Q_q)$
with the property that $i_q(\mathcal O_K\otimes_\Z\Z_q)\subset\M_2(\Z_q)$.
\sk

For each $r\geqslant 0$, let $R_{N^+,r}$ be the Eichler order of $B$ of level $N^+p^r$ with respect to the above isomorphisms
$\{i_q:B_q\simeq{\rm M}_2(\Q_q)\}_{q\nmid N^-}$, and let $U_{N^+,r}$ be the compact open
subgroup of $\widehat{R}_{N^+,r}^\times$ defined by
\[
U_{N^+,r}:=\left\{(x_q)_q\in\widehat{R}_{N^+,r}^\times\;\;\vert\;\;i_p(x_p)\equiv\mat 1*0*\pmod{p^r}\right\}.
\]

Consider the double coset spaces 
\begin{equation}\label{def:gross-curve}
\widetilde X_{N^+,r}=B^\times\big\backslash\bigl(\Hom_\Q(K,B)\times\widehat{B}^\times\bigr)\big/U_{N^+,r},
\end{equation}
where $b\in B^\times$ acts on $(\Psi,g)\in\Hom_\Q(K,B)\times\widehat B^\times$ by
\[
b\cdot(\Psi,g)=(b\Psi b^{-1},bg)
\]
and $U_{N^+,r}$ acts on $\widehat{B}^\times$ by right multiplication.
As is well-known (see e.g. \cite[\S{2.1}]{LV-MM}), $\widetilde X_{N^+,r}$ may be naturally identified with
the set of $K$-rational points of certain genus zero curves defined over $\Q$. Nonetheless, there
is a nontrivial Galois action on $\widetilde{X}_{N^+,r}$ defined as follows:
If $\sigma\in{\rm Gal}(K^{\rm ab}/K)$ and $P\in\widetilde X_{N^+,r}$
is the class of a pair $(\Psi,g)$, then
\[
P^\sigma:=[(\Psi,\widehat{\Psi}(a)g)],
\]
where $a\in K^\times\backslash\widehat{K}^\times$ is chosen so that ${\rm rec}_K(a)=\sigma$.
It will be convenient to extend this action to an action of $G_K:={\rm Gal}(\overline{\Q}/K)$
in the obvious manner.
\sk

Finally, we note that $\widetilde X_{N^+,r} $ is also equipped
with standard actions of $U_p$, Hecke operators $T_\ell$ for $\ell\nmid Np$, and diamond
operators $\langle d \rangle$ for $d\in(\Z/p^r\Z)^\times$ (see \cite[\S{2.4}]{LV-MM}, for example).

\subsection{Compatible systems of Heegner Points}\label{subsec:construct}

For each integer $c\geqslant 1$, let $\cO_c=\Z+c\cO_K$ be the order of $K$ of conductor $c$.

\begin{definition}
We say that a point $P\in\widetilde X_{N^+,r}$ is a \emph{Heegner point of conductor $c$}
if $P$ is the class of a pair $(\Psi,g)$ with
\[
\Psi(\cO_c)=\Psi(K)\cap(B\cap g\widehat{R}_{N^+,r}g^{-1})
\]
and
\[
\Psi_p((\cO_c\otimes\Z_p)^\times\cap(1+p^r\cO_K\otimes\Z_p)^\times)
=\Psi_p((\cO_c\otimes\Z_p)^\times)\cap g_pU_{N^+,r,p}g_p^{-1},
\]
where $U_{N^+,r,p}$ denotes the $p$-component of $U_{N^+,r}$.
\end{definition}

Fix a decomposition $N^+\cO_K=\mathfrak{N}^+\overline{\mathfrak{N}^+}$, and for each prime $q\neq p$ define
\begin{itemize}
\item{} $\varsigma_q=1$, if $q\nmid N^+$;
\item{} $\varsigma_q=\delta_K^{-1}\begin{pmatrix}\boldsymbol{\theta} & \overline{\boldsymbol{\theta}} \\ 1 & 1 \end{pmatrix}
\in{\rm GL}_2(K_{\mathfrak{q}})={\rm GL}_2(\Q_q)$, if
$q=\mathfrak{q}\overline{\mathfrak{q}}$ splits with $\mathfrak{q}\mid\mathfrak{N}^+$,
\end{itemize}
and for each $s\geqslant 0$, let
\begin{itemize}
\item{} $\varsigma_p^{(s)}=\begin{pmatrix}\boldsymbol{\theta}&-1\\1&0\end{pmatrix}\begin{pmatrix}p^s&0\\0&1\end{pmatrix}
\in{\rm GL}_2(K_{\mathfrak{p}})={\rm GL}_2(\Q_p)$,
if $p=\mathfrak{p}\overline{\mathfrak{p}}$ splits in $K$;
\item{}
$\varsigma_p^{(s)}=\begin{pmatrix}0&1\\-1&0\end{pmatrix}\begin{pmatrix}p^s&0\\0&1\end{pmatrix}$, if $p$ is inert in $K$.
\end{itemize}

Set $\varsigma^{(s)}:=\varsigma_p^{(s)}\prod_{q\neq p}\varsigma_q$, viewed as an element in
$\widehat{B}^\times$ via the isomorphisms $\{i_q:B_q\simeq M_2(\Q_q)\}_{q\nmid N^-}$ introduced in Section~\ref{subsec:Sh}.
Let $\imath_K:K\hookrightarrow B$ be the inclusion. Then one easily checks
(see \cite[Thm.~1.2]{cas-longo} and the references therein) that for all $n, r\geqslant 0$
the points
\[
\widetilde{P}_{p^n,r}^{}:=[(\imath_K,\varsigma^{(n+r)})]\in\widetilde X_{N^+,r}
\]
are Heegner point of conductor $p^{n+r}$ with the following properties:
\begin{itemize}
\item{} \emph{Field of definition}: $\widetilde{P}_{p^n,r}\in H^0(L_{p^n,r},\widetilde X_{N^+,r})$,
where $L_{p^n,r}:=H_{p^{n+r}}(\boldsymbol{\mu}_{p^r})$ and $H_{c}$ is the ring class field of $K$ of conductor $c$.
\item{} \emph{Galois equivariance}: For all $\sigma\in{\rm Gal}(L_{p^n,r}/H_{p^{n+r}})$, we have
\[
\widetilde{P}_{p^n,r}^\sigma=\langle\vartheta(\sigma)\rangle\cdot\widetilde{P}_{p^n,r},
\]
where $\vartheta:{\rm Gal}(L_{p^n,r}/H_{p^{n+r}})\rightarrow\Z_p^\times/\{\pm{1}\}$ is such that
$\vartheta^2=\varepsilon_{\rm cyc}$.
\item{} \emph{Horizontal compatibility}: If $r>1$, then
\[
\sum_{\sigma\in{\rm Gal}(L_{p^n,r}/L_{p^{n-1},r})}
\widetilde{\alpha}_r(\widetilde{P}_{p^{n},r}^{{\sigma}})
=U_p\cdot\widetilde{P}_{p^{n},r-1},
\]
where $\widetilde{\alpha}_r:\widetilde X_{N^+,r}\rightarrow\widetilde{X}_{N^+,{r-1}}$ is the map
induced by the inclusion $U_{N^+,r}\subseteq U_{N^+,r-1}$.
\item{} \emph{Vertical Compatibility}: If $n>0$, then
\[
\sum_{\sigma\in{\rm Gal}(L_{p^n,r}/L_{p^{n-1},r})}\widetilde{P}_{p^{n},r}^{{\sigma}}
=U_p\cdot\widetilde{P}_{p^{n-1},r}.
\]
\end{itemize}

\begin{remark}
We will only consider the points $\widetilde{P}_{p^n,r}$
for a fixed a value of $N^-$ (which amounts to fixing the quaternion algebra $B/\Q$),
but it will be fundamental to consider \emph{different} values of $N^+$, and the relations between
the corresponding $\widetilde{P}_{p^n,r}$ (which clearly depend on $N^+$) under various degeneracy maps.
\end{remark}

\subsection{Critical character}\label{subsec:crit}

Factor the $p$-adic cyclotomic character as
\[
\varepsilon_{\rm cyc}=\varepsilon_{\rm tame}\cdot\varepsilon_{\rm wild}:
G_\Q\longrightarrow\Z_p^\times\simeq\boldsymbol{\mu}_{p-1}\times(1+p\Z_p)
\]
and define the \emph{critical character} $\Theta:G_\Q\rightarrow\Z_p\pwseries{1+p\Z_p}^\times$ by
\begin{equation}\label{def:crit}
\Theta(\sigma)=
[\varepsilon^{1/2}_{\rm wild}(\sigma)],
\end{equation}
where $\varepsilon_{\rm wild}^{1/2}$ is the unique square-root of $\varepsilon_{\rm wild}$ taking values in $1+p\Z_p$,
and $[\cdot]:1+p\Z_p\rightarrow\Z_p\pwseries{1+p\Z_p}^\times$ is the map given by the inclusion as group-like elements.


\subsection{Big Heegner points}\label{subsec:bigHP}

Recall the Shimura curves $\widetilde X_{N^+,p^r}$ from Section~\ref{subsec:Sh}, and set
\[
\mathfrak{D}_{N^+,r}:=e^{\rm ord}({\rm Div}(\widetilde{X}_{N^+,r})\otimes_{\Z}\Z_p).
\]
By the Jacquet--Langlands correspondence, $\mathfrak{D}_{N^+,r}$ is naturally endowed with an action of
the Hecke algebra $\mathbb{T}_{N,r}^{N^-}$. Let $(\mathbb T_{N,r}^{N^-})^\dagger$ be the
free $\mathbb T_{N,r}^{N^-}$-module of rank one equipped with the Galois action via the inverse
of the critical character $\Theta$, and set
\[
\mathfrak{D}_{N^+,r}^\dagger:=\mathfrak{D}_{N^+,r}\otimes_{\mathbb{T}_{N,r}^{N^-}}(\mathbb T_{N,r}^{N^-})^\dagger.
\]

Let $\widetilde{P}_{p^{n},r}\in\widetilde{X}_{N^+,r}$ be the system of Heegner points of Section~\ref{subsec:construct}, and denote by $\mathcal{P}_{p^{n},r}^{}$ the image of
$e^{\rm ord}\widetilde{P}_{p^{n},r}^{}$ in $\mathfrak{D}_{N^+,r}$.
By the Galois equivariance of $\widetilde{P}_{p^{n},r}$ (see \cite[\S{7.1}]{LV-MM}),
we have
\[
\mathcal{P}_{p^{n},r}^\sigma=\Theta(\sigma)\cdot\mathcal{P}_{p^{n},r}
\]
for all $\sigma\in{\rm Gal}(L_{p^n,r}/H_{p^{n+r}})$, and hence
$\mathcal{P}_{p^{n},r}$ defines an element
\begin{equation}\label{eq:n,m-sigma}
\mathcal{P}_{p^n,r}\otimes\zeta_r\in H^0(H_{p^{n+r}},\mathfrak{D}_{N^+,r}^\dagger).
\end{equation}

In the next section we shall see how this system of points, for varying $n$ and $r$, can be used
to construct various anticyclotomic $p$-adic $L$-functions.

\section{Anticyclotomic $p$-adic $L$-functions}\label{sec:p-adicL}

\subsection{Multiplicity one}\label{subsec:periods}

Keep the notations introduced in Chapter~\ref{sec:heegner}.
For each integer $k\geqslant 2$, denote by $L_k(R)$ the set of polynomials of degree
less than or equal to $k-2$ with coefficients in a ring $R$, and define
\[
\mathfrak{J}_{N^+,r,k}:=e^{\rm ord}H_0(\widetilde{X}_{N^+,r},\mathcal{L}_k(\Z_p)),
\]
where $\mathcal{L}_k(\Z_p)$ is the local system on $\widetilde{X}_{N^+,r}$ associated with $L_k(\Z_p)$.
The module $\mathfrak{J}_{N^+,r,k}$ is endowed with an action of the Hecke algebra $\mathbb{T}_{N,r,k}^{N^-}$
and with perfect ``intersection pairing'':
\begin{equation}\label{eq:pairing}
\langle\;,\;\rangle_k:\mathfrak{J}_{N^+,r,k}\times\mathfrak{J}_{N^+,r,k}\longrightarrow\Q_p
\end{equation}
(see \cite[Eq.~(3.9)]{ChHs1}) with respect to which the Hecke operators are self-adjoint.

\begin{theorem}\label{thm:3.1.1}
Let $\mathfrak{m}$ be a maximal ideal of $\mathbb{T}^{N^-}_{N,r,k}$
whose residual representation is irreducible and satisfies Assumption~\SU. 
Then $(\mathfrak{J}_{N^+,r,k})_{\mathfrak{m}}$ 
is free of rank one over $(\mathbb{T}_{N,r,k}^{N^-})_{\mathfrak{m}}$.
In particular, there is a
$(\mathbb{T}^{N^-}_{N,r,k})_{\mathfrak{m}}$-module isomorphism
\[
(\mathfrak{J}_{N^+,r,k})_{\mathfrak{m}}\overset{\alpha_{N,r,k}}\simeq (\mathbb{T}_{N,r,k}^{N^-})_{\mathfrak{m}}.
\]
\end{theorem}

\begin{proof}
If $k=2$ and $r=1$, this follows by combining \cite[Thm.~6.2]{pollack-weston} and [\emph{loc.cit.}, Prop.~6.5].
The general case will be deduced from this case in Section~\ref{subsec:period-families} using Hida theory.
\end{proof}

Let $f\in S_k(\Gamma_{0,1}(N,p^r))$ be an $N^-$-new eigenform, and suppose that $\mathfrak{m}$
is the maximal ideal of $\mathbb{T}_{N,r,k}^{N^-}$ containing the kernel of the associated $\Z_p$-algebra homomorphism
\[
\pi_f:(\mathbb{T}_{N,r,k}^{N^-})_{\mathfrak{m}}\longrightarrow\cO,
\]
where $\cO$ is the the finite extension
of $\Z_p$ generated by the Fourier coefficients of $f$.
Composing $\pi_f$ with an isomorphism $\alpha_{N,r,k}$ as in Theorem~\ref{thm:3.1.1},
we obtain an $\cO$-valued functional
\[
\psi_f:(\mathfrak{J}_{N^+,r,k})_{\mathfrak{m}}\longrightarrow\cO.
\]
By the duality $(\ref{eq:pairing})$, the map $\psi_f$ corresponds to a generator $g_f$
of the $\pi_f$-isotypical component of $\mathfrak{J}_{N^+,r,k}$, and following \cite[\S{2.1}]{pollack-weston}
and \cite[\S{4.1}]{ChHs1} we define the \emph{Gross period} $\Omega_{f,N^-}$ attached to $f$ by
\begin{equation}\label{def:period}
\Omega_{f,N^-}:=\frac{(f,f)_{\Gamma_0(N)}}{\langle g_f,g_f\rangle_k}.
\end{equation}

\begin{remark}
By Vatsal's work \cite{Vat1} (see also \cite[Thm.~2.3]{pollack-weston} and \cite[\S{5.4}]{ChHs1}),
the anticyclotomic $p$-adic $L$-functions $L_p(f/K)$ in Theorem~\ref{thm:3.4.3} below
(normalized by the complex period $\Omega_{f,N^-}$) have vanishing $\mu$-invariant.
The preceding uniform description of $\psi_f$ for all $f$
with a common maximal ideal $\mathfrak{m}$ will allow us to show that this property is preserved
in Hida families.
\end{remark}

\subsection{One-variable $p$-adic $L$-functions}

Denote by $\Gamma$ the Galois group
of the anticyclotomic $\Z_p$-extension $K_\infty/K$. For each $n$, let $K_n\subset K_\infty$ be
defined by $\Gal(K_n/K)\simeq\Z/p^n\Z$ and let $\Gamma_n$ be the subgroup of $\Gamma$
such that $\Gamma/\Gamma_n\simeq\Gal(K_n/K)$.
\sk

Let $\mathcal{P}_{p^{n+1},r}^{}\otimes\zeta_r\in H^0(H_{p^{n+1+r}},\mathfrak{D}^\dagger_{N^+,r})$
be the Heegner point of conductor $p^{n+1}$, 
and define
\begin{equation}\label{def:Q}
\mathcal{Q}_{n,r}^{}:={\rm Cor}_{H_{p^{n+1+r}}/K_n}(\mathcal{P}_{p^{n+1},r}^{}\otimes\zeta_r)
\in H^0(K_n,\mathfrak{D}^\dagger_{N^+,r});
\end{equation}
with a slight abuse of notation, we also denote by $\mathcal{Q}_{n,r}^{}$ its image
under the natural map
\[
H^0(K_n,\mathfrak{D}^\dagger_{N^+,r})\overset{\subseteq}\longrightarrow\mathfrak{D}_{N^+,r}
\longrightarrow\mathfrak{J}_{N^+,r}
\]
composed with localization at $\mathfrak{m}$, where $\mathfrak{J}_{N^+,r}:=\mathfrak{J}_{N^+,r,2}$.

\begin{definition}\label{def:1var}
For any open subset $\sigma\Gamma_n$ of $\Gamma$, define
\[
\mu_r^{}(\sigma\Gamma_n):=U_p^{-n}\cdot\mathcal{Q}^\sigma_{n,r}
\in(\mathfrak{J}_{N^+,r})_{\mathfrak{m}}.
\]
\end{definition}

\begin{proposition}
The rule $\mu_r$ is a measure on $\Gamma$.
\end{proposition}

\begin{proof}
This follows immediately from the ``horizontal compatibility'' of Heegner points.
\end{proof}

%
%

\subsection{Gross periods in Hida families}\label{subsec:period-families}

Keep the notations of Section~\ref{subsec:periods}, and let
\[
(\mathfrak{J}_{N^+})_{\mathfrak{m}}:=\varprojlim_r(\mathfrak{J}_{N^+,r})_{\mathfrak{m}}
\]
which is naturally equipped with an action of the big Hecke algebra $\mathbb{T}_N^{N^-}=\varprojlim_r\mathbb{T}^{N^-}_{N,r}$.

\begin{theorem}\label{thm:3.3.1}
Let $\mathfrak{m}$ be a maximal ideal of $\mathbb{T}^{N^-}_{N}$ whose residual representation
is irreducible and satisfies Assumption~\SU.
Then $(\mathfrak{J}_{N^+})_{\mathfrak{m}}$ is free of rank one over $(\mathbb{T}_{N}^{N^-})_{\mathfrak{m}}$.
In particular, there is a $(\mathbb{T}^{N^-}_{N})_{\mathfrak{m}}$-module isomorphism
\[
(\mathfrak{J}_{N^+})_{\mathfrak{m}}\overset{\alpha_{N}}\simeq(\mathbb{T}_{N}^{N^-})_{\mathfrak{m}}.
\]
\end{theorem}

\begin{proof}
As in \cite[Prop.~3.3.1]{EPW}. Note that the version of Hida's control theorem in our context
is provided by \cite[Thm.~9.4]{hida-annals}.
\end{proof}

We can now conclude the proof of Theorem~\ref{thm:3.1.1} just as in \cite[\S{3.3}]{EPW}.
For the convenience of the reader, we include here the argument.

\begin{proof}[Proof of Theorem~\ref{thm:3.1.1}]
Let $\wp_{N,r,k}$ be the product of all the arithmetic primes of $\mathbb{T}^{N^-}_N$ of weight $k$ which
become trivial upon restriction to $1+p^r\Z_p$. By \cite[Thm.~9.4]{hida-annals}, we then have
\begin{equation}\label{control-H0}
(\mathfrak{J}_{N^+})_{\mathfrak{m}}\otimes\mathbb{T}^{N^-}_N/\wp_{N,r,k}
\simeq(\mathfrak{J}_{N^+,r,k})_{\mathfrak{m}_{r,k}}
\end{equation}
where $\mathfrak{m}_{r,k}$ is the maximal ideal of $\mathbb{T}^{N^-}_{N,r,k}$ induced by $\mathfrak{m}$.
Since $(\mathfrak{J}_{N^+})_{\mathfrak{m}}$ is free of rank one over $\mathbb{T}^{N^-}_N$ by Theorem~\ref{thm:3.3.1},
it follows that $(\mathfrak{J}_{N^+,r,k})_{\mathfrak{m}_{r,k}}$ is free of rank one
over $\mathbb{T}^{N^-}_N/\wp_{N,r,k}\simeq\mathbb{T}^{N^-}_{N,r,k}$, as was to be shown.
\end{proof}

\begin{remark}\label{remark3.5}
In the above proofs we made crucial use of \cite[Thm.~9.4]{hida-annals}, which is stated in the context
of totally definite quaternion algebras which are unramified at all finite places, since this is the only
relevant case for the study of Hilbert modular forms over totally real number fields of even degree. However,
the proofs immediately extend to the (simpler) situation of definite quaternion algebras over $\Q$.
\end{remark}

\subsection{Two-variable $p$-adic $L$-functions}\label{subsec:2varL}

By the ``vertical compatibility'' satisfied by Heegner points,
the points
\[
U_p^{-r}\cdot\mathcal{Q}_{n,r}\in(\mathfrak{J}_{N^+,r})_{\mathfrak{m}}
\]
are compatible for varying $r$, thus defining an element
\[
\mathcal{Q}_n^{}:=\varprojlim_r U_p^{-r}\cdot\mathcal{Q}_{n,r}^{}
\in(\mathfrak{J}_{N^+})_{\mathfrak{m}}.
\]

\begin{definition}\label{def:2var}
For any open subset $\sigma\Gamma_n$ of $\Gamma$, define
\[
\mu^{}(\sigma\Gamma_n):=U_p^{-n}\cdot\mathcal{Q}_{n}^{\sigma}
\in(\mathfrak{J}_{N^+})_{\mathfrak{m}}.
\]
\end{definition}

\begin{proposition}
The rule $\mu$ is a measure on $\Gamma$.
\end{proposition}

\begin{proof}
This follows immediately from the ``horizontal compatibility'' of Heegner points.
\end{proof}

Upon the choice of an isomorphism $\alpha_N$ as in Theorem~\ref{thm:3.3.1},
we may regard $\mu$ as an element
\[
\mathcal{L}(\mathfrak{m},N)\in(\mathbb{T}^{N^-}_N)_{\mathfrak{m}}\hat{\otimes}_{\Z_p}\Z_p\pwseries{\Gamma}.
\]

Denoting by $\mathcal{L}(\mathfrak{m},N)^*$ the image of $\mathcal{L}(\mathfrak{m},N)$ under the
involution induced by $\gamma\mapsto\gamma^{-1}$ on group-like elements, we set
\[
L(\mathfrak{m},N):=\mathcal{L}(\mathfrak{m},N)\cdot\mathcal{L}(\mathfrak{m},N)^*,
\]
to which we will refer as the \emph{two-variable $p$-adic $L$-function attached to $(\mathbb{T}^{N^-}_N)_{\mathfrak{m}}$}.

\subsection{Two-variable $p$-adic $L$-functions on branches of the Hida family}\label{subsec:2var-branch}

Let $\mathbb{T}_\Sigma$ be the
universal $p$-ordinary Hecke algebra
\begin{equation}\label{eq:2.4.2}
\mathbb{T}_\Sigma:=(\mathbb{T}_{N(\Sigma)}')_{\mathfrak{m}}\simeq(\mathbb{T}_{N(\Sigma)}^{N^-})_{\mathfrak{n}}
\end{equation}
associated with a mod $p$ representation $\bar{\rho}$ and a finite set of primes $\Sigma$ as in Section~\ref{subsec:residual}.

\begin{remark}\label{rem:split}
Recall that $N^-\mid N(\bar\rho)$ by Assumption~(SU). Throughout the following,
it will be further assumed that every prime factor of $N(\Sigma)/N^-$ splits in $K$. In particular, 
every prime $\ell\in\Sigma$ splits in $K$, and any $f\in\mathcal{H}^-(\bar\rho)={\rm Spec}(\mathbb{T}_\Sigma)$
has tame level $N_f$ with
\[
N_f^-=N(\bar\rho)^-=N^-.
\]
\end{remark}

The construction of the preceding section produces a two-variable $p$-adic $L$-function
\[
L(\mathfrak{n},N(\Sigma))\in(\mathbb{T}_{N(\Sigma)}^{N^-})_{\mathfrak{n}}\hat{\otimes}_{\Z_p}\Z_p\pwseries{\Gamma}
\]
which combined with the isomorphism $(\ref{eq:2.4.2})$ yields an element
\[
L_\Sigma(\bar{\rho})\in\mathbb{T}_\Sigma\hat{\otimes}_{\Z_p}\Z_p\pwseries{\Gamma}.
\]
If $\mathfrak{a}$ is a minimal prime of $\mathbb{T}_\Sigma$, we thus obtain an element
\[
L_\Sigma(\bar{\rho},\mathfrak{a})\in\mathbb{T}(\mathfrak{a})^\circ\hat{\otimes}_{\Z_p}\Z_p\pwseries{\Gamma}
\]
by reducing $L_\Sigma(\bar{\rho})$ mod $\mathfrak{a}$ (see \S\ref{sec:branches}). On the other hand, if we
let $\mathfrak{m}$ denote the inverse image of the maximal ideal of $\mathbb{T}(\mathfrak{a})^\circ$
under the composite surjection
\begin{equation}\label{eq:3.7}
\mathbb{T}^{N^-}_{N(\mathfrak{a})}\longrightarrow\mathbb{T}_{N(\mathfrak{a})}^{\rm new}
\longrightarrow\mathbb{T}_{N(\mathfrak{a})}^{\rm new}/\mathfrak{a}'=\mathbb{T}(\mathfrak{a})^\circ,
\end{equation}
the construction of the preceding section yields an $L$-function
\[
L(\mathfrak{m},N(\mathfrak{a}))\in(\mathbb{T}_{N(\mathfrak{a})}^{N^-})_{\mathfrak{m}}\hat{\otimes}_{\Z_p}\Z_p\pwseries{\Gamma}
\]
giving rise, via $(\ref{eq:3.7})$, to a second element
\[
L(\bar\rho,\mathfrak{a})\in\mathbb T(\mathfrak a)^\circ\hat{\otimes}_{\Z_p}\Z_p\pwseries{\Gamma}.
\]

It is natural to compare $L_\Sigma(\bar{\rho},\mathfrak{a})$ and $L(\bar{\rho},\mathfrak{a})$, a task that is
carried out in the next section, and provides the key for understanding the variation of
\emph{analytic} Iwasawa invariants.

\subsection{Comparison}\label{subsec:comparison}

Write $\Sigma=\{\ell_1,\dots,\ell_n\}$ and for each $\ell=\ell_i\in\Sigma$, let
$e_\ell$ be the valuation of $N(\Sigma)/N(\mathfrak a)$ at $\ell$,
and define the reciprocal Euler factor $E_{\ell}(\mathfrak{a},X)\in\mathbb T(\mathfrak a)^\circ[X]$ by
\[
E_{\ell}(\mathfrak{a},X):=
\begin{cases}
1& \text{ if $e_\ell=0$;}\\
1-(T_{\ell}\;{\rm mod}\;{\mathfrak{a}'})\Theta^{-1}(\ell)X&\text{ if $e_\ell=1$;}\\
1-(T_{\ell}\;{\rm mod}\;{\mathfrak{a}'})\Theta^{-1}(\ell)X
+\ell X^2&\text{ if $e_\ell=2$.}
\end{cases}
\]

Also, writing $\ell=\mathfrak{l}\bar{\mathfrak{l}}$,
define $E_\ell(\mathfrak{a})\in\mathbb T(\mathfrak a)^\circ\hat{\otimes}_{\Z_p}\Z_p\pwseries{\Gamma}$ by
\begin{equation}\label{def:e-a}
E_\ell(\mathfrak{a}):=E_\ell(\mathfrak{a},\ell^{-1}\gamma_{\mathfrak{l}})
\cdot E_\ell(\mathfrak{a},\ell^{-1}\gamma_{\bar{\mathfrak{l}}}),
\end{equation}
where $\gamma_{\mathfrak{l}}$, $\gamma_{\bar{\mathfrak{l}}}$ are arithmetic Frobenii at $\mathfrak{l}$, $\bar{\mathfrak{l}}$
in $\Gamma$, respectively, and put $E_\Sigma(\mathfrak a):=\prod_{\ell\in\Sigma}E_{\ell}(\mathfrak a)$.
\sk

Recall that $N^-\mid N(\mathfrak{a})\mid N(\Sigma)$ and set
\[
N(\mathfrak{a})^+:=N(\mathfrak{a})/N^-,\quad\quad N(\Sigma)^+:=N(\Sigma)/N^-,
\]
both of which consist entirely of prime factors 
which split in $K$. The purpose of this section is to prove the following result.

\begin{theorem}\label{thm:3.6.2}
There is an isomorphism of $\mathbb{T}(\mathfrak{a})^\circ$-modules
\[
\mathbb{T}(\mathfrak{a})^\circ\otimes_{(\mathbb{T}_{N(\Sigma)}^{N^-})_{\mathfrak{n}}}(\mathfrak{J}_{N(\Sigma)^+})_{\mathfrak{n}}\;\simeq\;
\mathbb{T}(\mathfrak{a})^\circ\otimes_{(\mathbb{T}_{N(\mathfrak{a})}^{N^-})_{\mathfrak{m}}}(\mathfrak{J}_{N(\mathfrak{a})^+})_{\mathfrak{m}}
\]
such that the map induced on the corresponding spaces of measures valued in these modules
sends $L_\Sigma(\bar\rho,\mathfrak{a})$ to $E_\Sigma(\mathfrak{a})\cdot L(\bar\rho,\mathfrak a)$.
\end{theorem}

\begin{proof}
The proof follows closely the constructions and arguments in \cite[\S{3.8}]{EPW}, suitably
adapted to the quaternionic setting at hand.
Let $r\geqslant 1$. If $M$ is any positive integer with $(M,pN^-)=1$, and $d'\mid d$ are divisors of $M$,
we have degeneracy maps
\[
B_{d,d'}:\widetilde X_{M,r}\longrightarrow\widetilde X_{M/d,r}
\]
induced by $(\Psi,g)\mapsto(\Psi,\pi_{d'}g)$, where
$\pi_{d'}\in\widehat{B}^\times$ has local component $\smallmat 100{\ell^{\val_\ell(d')}}$ at every prime $\ell\mid d'$
and $1$ outside $d'$. We thus obtain a map on homology
\[
(B_{d,d'})_*:e^{\rm ord}H_0(\widetilde X_{M,r},\Z_p)\longrightarrow e^{\rm ord}H_0(\widetilde X_{M/d,r},\Z_p)
\]
and we may define
\begin{equation}\label{eq:eps_r}
\epsilon_r:e^{\rm ord}H_0(\widetilde{X}_{N(\Sigma)^+,r},\Z_p)\longrightarrow e^{\rm ord}H_0(\widetilde{X}_{N(\mathfrak{a})^+,r},\Z_p)
\end{equation}
by $\epsilon_r:=\epsilon(\ell_n)\circ\cdots\circ\epsilon(\ell_1)$, where
for every $\ell=\ell_i\in\Sigma$ we put
\[
\epsilon(\ell):=
\begin{cases}1 & \textrm{if $e_\ell=0$;}\\
(B_{\ell,1})_*-(B_{\ell,\ell})_*\ell^{-1}T_\ell & \textrm{if $e_\ell=1$;}\\
(B_{\ell^2,1})_*-(B_{\ell^2,\ell})_*\ell^{-1}T_\ell+(B_{\ell^2,\ell^2})_*\ell^{-1}
\langle\ell\rangle_{N(\mathfrak{a})p} & \textrm{if $e_\ell=2$}.
\end{cases}
\]

As before, let $M$ be a positive integer with $(M,pN^-)=1$ all of whose prime factors split in $K$, and
let $\ell\nmid Mp$  be a prime which also splits in $K$. We shall adopt the following simplifying notations for the system of
points $\widetilde{P}_{p^n,r}\in\widetilde{X}_{N^+,r}$ constructed in Section~\ref{subsec:construct}:
\[
P:=\widetilde{P}_{p^n,r}^{(M)}\in\widetilde{X}_{M,r},\quad
P^{(\ell)}:=\widetilde{P}_{p^n,r}^{(M\ell)}\in\widetilde{X}_{M\ell,r},\quad
P^{(\ell^2)}:=\widetilde{P}_{p^n,r}^{(M\ell^2)}\in\widetilde{X}_{M\ell^2,r}.
\]

It is easy to check that we have the following relations in $\widetilde{X}_{M,r}$:
\begin{itemize}
\item
$(B_{\ell, 1})_*(P ^{(\ell)}) = P$
\item
$(B_{\ell,\ell})_*(P ^{(\ell)}) = P ^{\sigma_{{\mathfrak{l}}}}$
\item
$(B_{\ell^2, 1})_*({P ^{(\ell^2)}}) = P $
\item
$(B_{\ell^2,\ell})_*(P ^{(\ell^2)}) = P^{\sigma_{{\mathfrak{l}}}}$
\item
$(B_{\ell^2,\ell^2})_*(P ^{(\ell^2)}) = P^{ {\sigma_{{\mathfrak{l}}}^2}}$
\end{itemize}
where $\sigma_{\mathfrak{l}}\in{\rm Gal}(L_{p^n,r}/K)$ is a Frobenius element at a prime $\mathfrak{l}\mid\ell$.
Letting $\mathcal{P}$ denote the image of $e^{\rm ord}P$ in $\mathfrak{D}_{M,r}$, and defining
$\mathcal{P}^{(\ell)}\in\mathfrak{D}_{M\ell,r}$ and $\mathcal{P}^{(\ell^2)}\in\mathfrak{D}_{M\ell^2,r}$
similarly, 
it follows that
\begin{itemize}
\item $(B_{\ell, 1})_*(\mathcal{P}^{(\ell)}\otimes\zeta_r)=
\mathcal P\otimes\zeta_r$
\item $(B_{\ell,\ell})_*(\mathcal P^{(\ell)}\otimes\zeta_r)=
\mathcal P^{\sigma_{{\mathfrak{l}}}}\otimes\zeta_r=
\Theta^{-1}(\sigma_\mathfrak{l})\cdot(\mathcal P\otimes\zeta_r)^{\sigma_\mathfrak{l}}$
\item $(B_{\ell^2, 1})_*(\mathcal P ^{(\ell^2)}\otimes\zeta_r) =
\mathcal P\otimes\zeta_r$
\item $(B_{\ell^2,\ell})_*(\mathcal P ^{(\ell^2)}\otimes\zeta_r) =
\mathcal P ^{\sigma_{{\mathfrak{l}}}}\otimes\zeta_r=
\Theta^{-1}(\sigma_\mathfrak{l})\cdot(\mathcal P\otimes\zeta_r)^{\sigma_\mathfrak{l}}$
\item $(B_{\ell^2,\ell^2})_*(\mathcal P ^{(\ell^2)}\otimes\zeta_r)
=\mathcal P^{\sigma^2_{{\mathfrak{l}}}}\otimes\zeta_r=
\Theta^{-2}(\sigma_\mathfrak{l})\cdot(\mathcal P \otimes\zeta_r)^{\sigma_\mathfrak{l}}$
\end{itemize}
as elements in $\mathfrak{D}_{M,r}^\dagger$. Finally, setting
$\mathcal{Q}_{}:={\rm Cor}_{H_{p^{n+1+r}}/K_n}(\mathcal{P})\in H^0(K_n,\mathfrak{D}_{M,r}^\dagger)$,
and defining $\mathcal{Q}_{}^{(\ell)}\in H^0(K_n,\mathfrak{D}_{M\ell,r}^\dagger)$ and
$\mathcal{Q}_{}^{(\ell^2)}\in H^0(K_n,\mathfrak{D}_{M\ell^2,r}^\dagger)$ similarly, we see that
\begin{itemize}
\item $(B_{\ell, 1})_*(\mathcal{Q}^{(\ell)})=\mathcal Q$
\item $(B_{\ell,\ell})_*(\mathcal{Q}^{(\ell)})=
\Theta^{-1}(\sigma_{\mathfrak{l}})\cdot\mathcal Q^{\sigma_{\mathfrak{l}}}$
\item $(B_{\ell^2, 1})_*({\mathcal Q ^{(\ell^2)})}=\mathcal Q$
\item $(B_{\ell^2,\ell})_*(\mathcal Q ^{(\ell^2)})=\Theta^{-1}(\sigma_\mathfrak{l})\cdot\mathcal Q^{\sigma_{\mathfrak{l}}}$
\item $(B_{\ell^2,\ell^2})_*(\mathcal Q ^{(\ell^2)})=\Theta^{-2}(\sigma_\mathfrak{l})\cdot\mathcal Q^{\sigma^2_{\mathfrak{l}}}$
\end{itemize}
in $H^0(K_n,\mathfrak{D}_{M,r}^\dagger)$. Each of these equalities is checked by an explicit calculation.
For example, for the second one:
\begin{align*}
(B_{\ell,\ell})_*(\mathcal Q^{(\ell)})
&=(B_{\ell,\ell})_*\left({\rm Cor}_{H_{p^{n+1+r}}/K_n}(\mathcal P^{(\ell)}\otimes\zeta_r)\right)\\
&=(B_{\ell,\ell})_*\left(\Bigl(\sum_{\sigma\in{\rm Gal}(H_{p^{n+1+r}}/K_n)}\Theta(\tilde{\sigma}^{-1})\cdot
(\mathcal P^{(\ell)})^{\tilde\sigma}\Bigr)\otimes\zeta_r\right)\\
&=\sum_{\sigma\in{\rm Gal}(H_{p^{n+1+r}}/K_n)}\Theta(\tilde\sigma^{-1})
\cdot(B_{\ell,\ell})_*((\mathcal P^{(\ell)})^{\tilde{\sigma}}\otimes\zeta_r)\\
&=\sum_{\sigma\in{\rm Gal}(H_{p^{n+1+r}}/K_n)}\Theta(\tilde\sigma^{-1})
\Theta^{-1}(\sigma_\mathfrak{l})\cdot(\mathcal P^{\tilde{\sigma}}\otimes\zeta_r)^{\sigma_{\mathfrak{l}}}\\
&=\Theta^{-1}(\sigma_\mathfrak{l})\cdot \mathcal Q^{\sigma_\mathfrak{l}}.
\end{align*}

Now let $\mathcal{Q}_{n,r}\in\mathfrak{J}_{N(\Sigma)^+,r}$ be as in $(\ref{def:Q})$ with $N=N(\Sigma)$.
Using the above formulae, we easily see that of any finite order character $\chi$ of
$\Gamma$ of conductor $p^n$, the effect of $\epsilon_r$ on the element
$\sum_{\sigma\in\Gamma/\Gamma_n}\chi(\sigma)\mathcal{Q}_{n,r}^\sigma$ is given by multiplication by
\[
\prod_{e_{\ell_i}=1}
(1-(\chi\Theta)^{-1}(\sigma_{\mathfrak{l}_i})\ell_i^{-1}T_{\ell_i})
\prod_{e_{\ell_i}=2}(1-(\chi\Theta)^{-1}(\sigma_{\mathfrak{l}_i})\ell_i^{-1}T_{\ell_i}
+(\chi\Theta)^{-2}(\sigma_{\mathfrak{l}_i})\ell_i^{-1}\langle\ell_i\rangle_{N(\mathfrak{a})p}).
\]
Similarly, we see that $\epsilon_r$ has the effect of multiplying the element
$\sum_{\sigma\in\Gamma/\Gamma_n}\chi^{-1}(\sigma)\mathcal{Q}_{n,r}^\sigma$ by
\[
\prod_{e_{\ell_i}=1}
(1-(\chi^{-1}\Theta)^{-1}(\sigma_{\mathfrak{l}_i})\ell_i^{-1}T_{\ell_i})
\prod_{e_{\ell_i}=2}(1-(\chi^{-1}\Theta)^{-1}(\sigma_{\mathfrak{l}_i})\ell_i^{-1}T_{\ell_i}
+(\chi^{-1}\Theta)^{-2}(\sigma_{\mathfrak{l}_i})\ell_i^{-1}\langle\ell_i\rangle_{N(\mathfrak{a})p}).
\]
Hence, using the relations
\[
\chi(\sigma_{\bar{\mathfrak{l}}_i})=\chi^{-1}(\sigma_{\mathfrak{l}_i}),
\quad\quad\Theta(\sigma_{\mathfrak{l}_i})=\Theta(\sigma_{\bar{\mathfrak{l}}_i})=\theta(\ell_i),
\quad\quad\theta^2(\ell_i)=\langle\ell_i\rangle_{N(\mathfrak{a})p},
\]
it follows that the effect of $\epsilon_r$ on the product of the above two elements
is given by multiplication by
\[
\prod_{\substack{\mathfrak{l}_i\mid \ell_i\\e_{\ell_i}=1}}(1-\chi(\sigma_{\mathfrak{l}_i})\theta^{-1}(\ell_i)\ell_i^{-1}T_{\ell_i})
\prod_{\substack{\mathfrak{l}_i\mid \ell_i\\e_{\ell_i}=2}}(1-\chi(\sigma_{\mathfrak{l}_i})\theta^{-1}(\ell_i)\ell_i^{-1}T_{\ell_i}+
\chi^{2}(\sigma_{\mathfrak{l}_i})\ell_i^{-1}).
\]

Taking the limit over $r$, 
we thus obtain 
a $\mathbb{T}(\mathfrak{a})^\circ$-linear map
\begin{equation}\label{3.11}
\mathbb{T}(\mathfrak{a})^\circ\otimes_{(\mathbb{T}_{N(\Sigma)}^{N^-})_{\mathfrak{n}}}(\mathfrak{J}_{N(\Sigma)^+})_{\mathfrak{n}}\longrightarrow
\mathbb{T}(\mathfrak{a})^\circ\otimes_{(\mathbb{T}_{N(\mathfrak{a})}^{N^-})_{\mathfrak{m}}}(\mathfrak{J}_{N(\mathfrak{a})^+})_{\mathfrak{m}}
\end{equation}
having as effect on the corresponding measures as stated in Theorem~\ref{thm:3.6.2}.
Hence to conclude the proof it remains to show that $(\ref{3.11})$ is an isomorphism.

By Theorem~\ref{thm:3.3.1}, both the source and the target of this map are free of rank one over $\mathbb{T}(\mathfrak{a})^\circ$,
and as in \cite[p.559]{EPW} (using \cite[Thm.~9.4]{hida-annals}), one is reduced to showing
the injectivity of the dual map modulo $p$:
\begin{equation}\label{3.15}
\begin{split}
H^0(\widetilde{X}_{N(\mathfrak{a})^+,1};\mathbf{F}_p)^{\rm ord}[\mathfrak{m}]
&\longrightarrow(\mathbb{T}_{N(\mathfrak{a})}^{N^-}/\mathfrak{m})
\otimes_{\mathbb{T}_{N(\Sigma)}^{N^-}/\mathfrak{n}}(H^0(\widetilde{X}_{N(\mathfrak{a})^+,1};\mathbf{F}_p)^{\rm ord}[\mathfrak{m}'])\\
&\longrightarrow(\mathbb{T}_{N(\mathfrak{a})}^{N^-}/\mathfrak{m})
\otimes_{\mathbb{T}_{N(\Sigma)}^{N^-}/\mathfrak{n}}(H^0(\widetilde{X}_{N(\Sigma)^+,1};\mathbf{F}_p)^{\rm ord}[\mathfrak{m}'])\\
&\longrightarrow(\mathbb{T}_{N(\mathfrak{a})}^{N^-}/\mathfrak{m})
\otimes_{\mathbb{T}_{N(\Sigma)}^{N^-}/\mathfrak{n}}(H^0(\widetilde{X}_{N(\Sigma)^+,1};\mathbf{F}_p)^{\rm ord}[\mathfrak{n}]);
\end{split}
\end{equation}
or equivalently (by a version of \cite[Lemma~3.8.1]{EPW}), to showing that the composite of the first
two arrows in (\ref{3.15}) is injective.

In turn, the latter injectivity follows from Lemma~\ref{lem:3.8.2} below, where the notations are as follows:
\begin{itemize}
\item{} $M^+$ is any positive integer with $(M^+,pN^-)=1$;
\item{} $\ell\neq p$ is a prime;
\item{} $n_\ell=1$ or $2$ according to whether or not $\ell$ divides $M^+$;
\item{} $N^+:=\ell^{n_\ell}M^+$,
\end{itemize}
and
\begin{equation}\label{3.17}
\epsilon_\ell^*:H^0(\widetilde{X}_{M^+,1};\mathbf{F}_p)^{\rm ord}[\mathfrak{m}]\longrightarrow(\mathbb{T}^{N^-}_{M^+N^-}/\mathfrak{m})
\otimes_{\mathbb{T}_{N^+N^-}'/\mathfrak{m}'}(H^0(\widetilde{X}_{N^+,1};\mathbf{F}_p)^{\rm ord}[\mathfrak{m}'])
\end{equation}
is the map defined by
\[
\epsilon_\ell^*:=
\left\{
\begin{array}{ll}
B_{\ell,1}^*-B_{\ell,\ell}^*\ell^{-1}T_\ell & \textrm{if $n_\ell=1$;}\\
B_{\ell^2,1}^*-B_{\ell^2,\ell}^*\ell^{-1}T_\ell+B_{\ell^2,\ell^2}^*\ell^{-1}
\langle\ell\rangle_{N(\mathfrak{a})p} & \textrm{if $n_\ell=2$}.
\end{array}
\right.
\]

\begin{lemma}\label{lem:3.8.2}
The map $(\ref{3.17})$ is injective.
\end{lemma}

\begin{proof}
As in the proof of the analogous result \cite[Lemma~3.8.2]{EPW} in the modular curve case,
it suffices to show the injectivity of the map
\[
(H^0(\widetilde{X}_{M^+,1};\mathbf{F})^{\rm ord}[\mathfrak{m}_{\mathbf{F}}])^{n_\ell+1}\xrightarrow{\;\beta_\ell\;}
H^0(\widetilde{X}_{N^+,1};\mathbf{F})^{\rm ord}[\mathfrak{m}_{\mathbf{F}}']
\]
defined by
\[
\beta_\ell:=
\left\{
\begin{array}{ll}
B_{\ell,1}^*\pi_1+B_{\ell,\ell}^*\pi_2 & \textrm{if $n_\ell=1$;}\\
B_{\ell^2,1}^*\pi_1+B_{\ell^2,\ell}^*\pi_2+B_{\ell^2,\ell^2}^*\pi_3 & \textrm{if $n_\ell=2$}.
\end{array}
\right.
\]

But in our quaternionic setting the proof of this injectivity follows from \cite[Lemma~3.26]{SW} for $n_\ell=1$
and [\emph{loc.cit.}, Lemma~3.28] for $n_\ell=2$.
\end{proof}

Applying inductively Lemma~\ref{lem:3.8.2} to the primes in $\Sigma$,
the proof of Theorem~\ref{thm:3.6.2} follows.
\end{proof}

\subsection{Analytic Iwasawa invariants}

Upon the choice of an isomorphism
\[
\Z_p\pwseries{\Gamma}\simeq\Z_p\pwseries{T}
\]
we may regard the $p$-adic $L$-functions $L_\Sigma(\bar\rho,\mathfrak{a})$ and $L(\bar\rho,\mathfrak{a})$,
as well as the Euler factor $E_\Sigma(\bar\rho,\mathfrak{a})$, as elements
in $\mathbb{T}(\mathfrak{a})^\circ\pwseries{T}$. In this section we apply the main result of the
preceding section to study the variation of the Iwasawa invariants attached to the anticyclotomic
$p$-adic $L$-functions of $p$-ordinary modular forms.
\sk

For any power series $f(T)\in R\pwseries{T}$ with coefficients in a ring $R$,
the \emph{content} of $f(T)$ is defined to be the ideal $I(f(T))\subseteq R$ generated
by the coefficients of $f(T)$. If $\wp$ is a height one prime of $\mathbb{T}_\Sigma$ belonging to the branch $\mathbb T(\mathfrak a)$
(in the sense that $\mathfrak{a}$ is the unique minimal prime of $\mathbb{T}_\Sigma$ contained in $\wp$),
we denote by $L(\bar\rho,\mathfrak a)(\wp)$ the element of $\mathcal O_\wp\pwseries{\Gamma}$
obtained from $L(\bar\rho,\mathfrak a)$ via reduction modulo $\wp$. In particular, we note that
$L(\bar\rho,\mathfrak{a})(\wp)$ has unit content if and only if its $\mu$-invariant vanishes.

\begin{theorem}\label{3.7.5}
The following are equivalent:
\begin{enumerate}
\item $\mu(L(\bar\rho,\mathfrak a)(\wp))=0$ for some newform $f_\wp$ in $\mathcal{H}^-(\bar\rho)$.
\item $\mu(L(\bar\rho,\mathfrak a)(\wp))=0$ for every newform $f_\wp$ in $\mathcal{H}^-(\bar\rho)$.
\item $L(\bar{\rho},\mathfrak{a})$ has unit content for some irreducible component $\mathbb{T}(\mathfrak{a})$ of $\mathcal{H}^-(\bar\rho)$.
\item $L(\bar{\rho},\mathfrak{a})$ has unit content for every irreducible component $\mathbb{T}(\mathfrak{a})$ of $\mathcal{H}^-(\bar\rho)$.
\end{enumerate}
\end{theorem}

\begin{proof}
The argument in \cite[Thm~3.7.5]{EPW} applies verbatim, replacing the appeal to [\emph{loc.cit.}, Cor.~3.6.3]
by our Theorem~\ref{thm:3.6.2} above.
\end{proof}

When any of the conditions in Theorem~\ref{3.7.5} hold,
we shall write
\[
\mu^\mathrm{an}(\bar\rho)=0.
\]
For a power series $f(T)$ with unit content
and coefficients in a local ring $R$, 
the $\lambda$-invariant $\lambda(f(T))$ is defined to be the smallest degree in which $f(T)$ has a unit coefficient.



\begin{theorem}\label{3.7.7}
Assume that $\mu^\mathrm{an}(\bar\rho)=0$.
\begin{enumerate}
\item Let $\mathbb T(\mathfrak a)$ be an irreducible component of $\mathcal{H}^-(\bar\rho)$.
As $\wp$ varies over the arithmetic primes of $\mathbb T(\mathfrak a)$, the $\lambda$-invariant
$\lambda(L(\bar\rho,\mathfrak a)(\wp))$ takes on a constant value, denoted $\lambda^{\rm an}(\bar\rho,\mathfrak{a})$.
\item For any two irreducible components $\mathbb T(\mathfrak a_1), \mathbb T(\mathfrak a_2)$ of $\mathcal{H}^-(\bar\rho)$,
we have that
\[
\lambda^\mathrm{an}(\bar\rho,\mathfrak a_1)-\lambda^\mathrm{an}(\bar\rho,\mathfrak a_2)=
\sum_{\ell\neq p}e_\ell(\mathfrak a_2)-e_\ell(\mathfrak a_1),
\]
where $e_\ell(\mathfrak a)=\lambda(E_\ell(\mathfrak{a}))$.
\end{enumerate}
\end{theorem}

\begin{proof}
The first part follows immediately from the definitions. 
For the second part, the argument in \cite[Thm.~3.7.7]{EPW} applies verbatim, replacing their
appeal to [\emph{loc.cit.}, Cor.~3.6.3] by our Theorem~\ref{thm:3.6.2} above.
\end{proof}

By Theorem~\ref{3.7.5} and Theorem~\ref{3.7.7}, the Iwasawa invariants of $L(\bar\rho,\mathfrak a)(\wp)$
are well-behaved as $\wp$ varies. However, for the applications of these results to the Iwasawa main conjecture
it is of course necessary to relate $L(\bar\rho,\mathfrak a)(\wp)$ to
$p$-adic $L$-functions defined by the interpolation of special values of $L$-functions.
This question was addressed in 
\cite{cas-longo}, as we now recall.

\begin{theorem}\label{thm:3.4.3}
If $\wp$ is the arithmetic prime of $\mathbb T(\mathfrak a)$
corresponding to a $p$-ordinary $p$-stabilized newform $f_\wp$ of weight $k\geqslant 2$
and trivial nebentypus, then
\[
L(\bar\rho,\mathfrak a)(\wp)=L_p(f_\wp/K),
\]
where $L_p(f_{\wp}/K)$ is the $p$-adic $L$-function
of Chida--Hsieh \cite{ChHs1}. In particular, if $\chi:\Gamma\rightarrow\bC_p^\times$
is the $p$-adic avatar of an anticyclotomic Hecke character of $K$ of infinity type $(m,-m)$ with $-k/2<m<k/2$,
then $L(\bar\rho,\mathfrak a)(\wp)$ interpolates the central critical values
\[
\frac{L(f_\wp/K,\chi,k/2)}{\Omega_{f_\wp,N^-}}
\]
as $\chi$ varies, where $\Omega_{f_\wp,N^-}$ is the complex Gross period $(\ref{def:period})$.
\end{theorem}

\begin{proof}
This is a reformulation of the main result of \cite{cas-longo}.
(Note that the constant $\lambda_\wp\in F_\wp^\times$ in \cite{cas-longo}, Thm.~4.6] is not
needed here, since the specialization map of [\emph{loc.cit.}, \S{3.1}] is being replaced by the
map $(\mathfrak{J}_{N^+})_{\mathfrak{m}}\rightarrow(\mathfrak{J}_{N^+,r,k})_{\mathfrak{m}_{r,k}}$ induced by the
isomorphism $(\ref{control-H0})$, which preserves integrality.)
\end{proof}

\begin{corollary}\label{cor:lambda-an}
Let $f_1$, $f_2\in\mathcal{H}^-(\bar\rho)$ be newforms with trivial nebentypus lying in the branches
$\mathbb{T}(\mathfrak{a}_1)$, $\mathbb{T}(\mathfrak{a}_2)$, respectively.
Then $\mu^{\rm an}(\bar\rho)=0$ and
\[
\lambda(L_p(f_1/K))-\lambda(L_p(f_2/K))=
\sum_{\ell\neq p}e_\ell(\mathfrak{a}_2)-e_\ell(\mathfrak{a}_1),
\]
where $e_\ell(\mathfrak{a}_j)=\lambda(E_\ell(\mathfrak{a}_j))$.
\end{corollary}

\begin{proof}
By \cite[Thm.~5.7]{ChHs1} (extending Vatsal's result \cite{Vat1} to higher weights),
if $f\in\mathcal{H}^-(\bar\rho)$ has weight $k\leqslant p+1$ and trivial nebentypus,
then $\mu(L_p(f/K))=0$. By Theorem~\ref{3.7.5} and Theorem~\ref{thm:3.4.3}, this implies $\mu^{\rm an}(\bar\rho)=0$.
The result thus follows from Theorem~\ref{3.7.7}, using again Theorem~\ref{thm:3.4.3}
to replace $\lambda^{\rm an}(\bar\rho,\mathfrak{a}_j)$ by $\lambda(L_p(f_j/K))$.
\end{proof}

\section{Anticyclotomic Selmer groups}\label{sec:Selmer}

We continue with the notation of the previous sections. In particular,
$\bar\rho:G_\Q\rightarrow{\rm GL}_2(\mathbf{F})$ is an odd irreducible
Galois representation satisfying Assumption~\SU~and isomorphic to $\bar\rho_{f_0}$
for some newform $f_0$ of weight $2$, $\mathcal{H}^-(\bar\rho)$ is
the associated Hida family, and $\Sigma$ is a finite set of primes split in the
imaginary quadratic field $K$. 
\sk

For each $f\in\mathcal{H}^-(\bar\rho)$, let $V_f$
denote the self-dual Tate twist of the $p$-adic Galois representation associated to $f$,
fix an $\cO$-stable lattice $T_f\subseteq V_f$, and set $A_f:=V_f/T_f$.
Since $f$ is $p$-ordinary, there is a unique one-dimensional $G_{\bQ_p}$-invariant subspace
$F_p^+V_f\subseteq V_f$ where the inertia group at $p$ acts via $\varepsilon_{\rm cyc}^{k/2}\psi$,
with $\psi$ a finite order character. Let $F_p^+A_f$ be the image of $F_p^+V_f$ in $A_f$,
and as recalled in the Introduction define the \emph{minimal Selmer group} of $f$ by
\begin{equation}\label{def:Sel-min}
{\rm Sel}(K_\infty,f):=\ker\left\{H^1(K_\infty,A_f)
\longrightarrow
\prod_{w\nmid p}H^1(K_{\infty,w},A_f)\times\prod_{w\vert p}H^1(K_{\infty,w},F_p^-A_f)\right\},\nonumber
\end{equation}
where $w$ runs over the places of $K_\infty$ and we set $F_p^-A_f:=A_f/F_p^+A_f$. 
\sk

It is well-known that ${\rm Sel}(K_\infty,f)$ is cofinitely generated over $\Lambda$. When
it is also $\Lambda$-cotorsion, we define the $\mu$-invariant $\mu({\rm Sel}(K_\infty,f))$
(resp. $\lambda$-invariant $\lambda({\rm Sel}(K_\infty,f))$) to the largest power of $\varpi$ dividing
(resp. the number of zeros of) the characteristic power series of the Pontryagin dual of ${\rm Sel}(K_\infty,f)$.
\sk

A distinguishing feature of the anticyclotomic setting (in comparison with cyclotomic Iwasawa theory)
is the presence of primes which split infinitely in the corresponding $\Z_p$-extension. Indeed,
being inert in $K$, all primes $\ell\mid N^-$ are infinitely split in $K_\infty/K$. As a result,
the above Selmer group differs in general from the \emph{Greenberg Selmer group} of $f$, defined as
\begin{equation}\label{def:Sel-Gr}
\mathfrak{Sel}(K_\infty,f):=\ker\left\{H^1(K_\infty,A_f)
\longrightarrow
\prod_{w\nmid p}H^1(I_{\infty,w},A_f)\times\prod_{w\vert p}H^1(K_{\infty,w},F_p^-A_f)\right\},\nonumber
\end{equation}
where $I_{\infty,w}\subseteq G_{K_\infty}$ denotes the inertia group at $w$.
\sk

If $S$ is a finite set of primes in $K$,
we let ${\rm Sel}^S(K_\infty,f)$ and $\mathfrak{Sel}^S(K_\infty,f)$ be the ``$S$-primitive'' Selmer groups
defined as above by omitting the local conditions at the primes in $S$ (except those above $p$,
when any such prime is in $S$). Moreover, if $S$ consists of the primes
dividing a rational integer $M$, we replace the superscript $S$ by $M$ in the above notation.
\sk

Immediately from the definitions, we see that there is as exact sequence
\begin{equation}\label{eq:defs}
0\longrightarrow{\rm Sel}(K_\infty,f)\longrightarrow\mathfrak{Sel}(K_\infty,f)
\longrightarrow\prod_{\ell\vert N^-}\mathcal{H}^{\rm un}_\ell,
\end{equation}
where
\[
\mathcal{H}_\ell^{\rm un}:={\ker}\left\{\prod_{w\vert\ell}H^1(K_{\infty,w},A_f)
\longrightarrow\prod_{w\vert\ell}H^1(I_{\infty,w},A_f)\right\}
\]
is the set of unramified cocycles. In \cite[\S\S{3}, 5]{pollack-weston},
Pollack and Weston carried out a careful analysis of the difference
between ${\rm Sel}(K_\infty,f)$ and $\mathfrak{Sel}(K_\infty,f)$.
Even though \emph{loc.cit.} is mostly concerned with cases in which $f$ is of weight $2$,
many of their arguments apply more generally.
In fact, the next result follows essentially from their work.


\begin{theorem}\label{thm:mu-alg}
Assume that $\bar\rho$ satisfies Hypotheses~\SU. Then the following are equivalent:
\begin{enumerate}
\item{} ${\rm Sel}(K_\infty,f_0)$ is $\Lambda$-cotorsion with $\mu$-invariant zero
for some newform $f_0\in\mathcal{H}^-(\bar\rho)$.
\item{} ${\rm Sel}(K_\infty,f)$ is $\Lambda$-cotorsion with $\mu$-invariant zero
for all newforms $f\in\mathcal{H}^-(\bar\rho)$.
\item{} $\mathfrak{Sel}(K_\infty,f)$ is $\Lambda$-cotorsion with $\mu$-invariant zero
for all newforms $f\in\mathcal{H}^-(\bar\rho)$.
\end{enumerate}
Moreover, in that case ${\rm Sel}(K_\infty,f)\simeq\mathfrak{Sel}(K_\infty,f)$.
\end{theorem}
\begin{proof}
Assume $f_0$ is a newform in $\mathcal{H}^-(\bar\rho)$ for which
${\rm Sel}(K_\infty,f_0)$ is $\Lambda$-cotorsion with $\mu$-invariant zero, and set $N^+:=N(\Sigma)/N^-$.
By \cite[Prop.~5.1]{pollack-weston}, we then have the exact sequences
\begin{align}
0\longrightarrow{\rm Sel}(K_\infty,f_0)\longrightarrow &\;{\rm Sel}^{N^+}(K_\infty,f_0)
\longrightarrow\prod_{\ell\mid N^+}\mathcal{H}_\ell\longrightarrow 0;\label{5.1a}\\
0\longrightarrow\mathfrak{Sel}(K_\infty,f_0)\longrightarrow&\;\mathfrak{Sel}^{N^+}(K_\infty,f_0)
\longrightarrow\prod_{\ell\mid N^+}\mathcal{H}_\ell\longrightarrow 0,\label{5.1b}
\end{align}
where $\mathcal{H}_\ell$ is the product of $H^1(K_{\infty,w},A_{f_0})$
over the places $w\mid\ell$ in $K_\infty$.
Since every prime $\ell\mid N^+$ splits in $K$ (see Remark~\ref{rem:split}),
the $\Lambda$-cotorsionness and the vanishing of the $\mu$-invariant of $\mathcal{H}_\ell$
can be deduced from \cite[Prop.~2.4]{GV}. Since ${\rm Sel}(K_\infty,f_0)[\varpi]$
is finite by assumption, it thus follows from $(\ref{5.1a})$
that ${\rm Sel}^{N^+}(K_\infty,f_0)[\varpi]$ is finite. Combined with
$(\ref{eq:defs})$ and \cite[Cor.~5.2]{pollack-weston},
the same argument using (\ref{5.1b}) shows that then
$\mathfrak{Sel}^{N^+}(K_\infty,f_0)[\varpi]$ is also finite.

On the other hand, following the arguments in the proof \cite[Prop.~3.6]{pollack-weston}
we see that for any $f\in\mathcal{H}(\bar{\rho})$ we have the isomorphisms
\begin{align*}
{\rm Sel}^{N^+}(K_\infty,\bar{\rho})
&\simeq{\rm Sel}^{N^+}(K_\infty,f)[\varpi];\\
\mathfrak{Sel}^{N^+}(K_\infty,\bar{\rho})
&\simeq\mathfrak{Sel}^{N^+}(K_\infty,f)[\varpi].
\end{align*}
As a result, the argument in the previous paragraph implies that, for any newform $f\in\mathcal{H}^-(\bar\rho)$,
both ${\rm Sel}^{N^+}(K_\infty,f)[\varpi]$ and $\mathfrak{Sel}^{N^+}(K_\infty,f)[\varpi]$ are finite ,
from where (using (\ref{5.1a}) and (\ref{5.1b}) with $f$ in place of $f_0$)
the $\Lambda$-cotorsionness and the vanishing of both the $\mu$-invariant
of ${\rm Sel}(K_\infty,f)$ and of $\mathfrak{Sel}(K_\infty,f)$ follows.
In view of $(\ref{eq:defs})$ and \cite[Lemma~3.4]{pollack-weston},
the result follows.
\end{proof}

Let $w$ be a prime of $K_\infty$ above $\ell\neq p$ and denote by $G_{w}\subseteq G_{K_\infty}$ its decomposition group.
Let $\mathbb T(\mathfrak a)$ be the irreducible component of $\mathbb{T}_\Sigma$ passing through $f$, and define 
\[
\delta_w(\mathfrak a):=\dim_{\mathbf{F}}A_f^{G_{w}}/\varpi.
\]
(Note that this is well-defined by \cite[Lemma 4.3.1]{EPW}.)
Assume $\ell=\mathfrak{l}\bar{\mathfrak{l}}$ splits in $K$ and put
\begin{equation}\label{def:delta-ell}
\delta_\ell(\mathfrak a)
:=\sum_{w\mid \ell}\delta_w(\mathfrak a),
\end{equation}
where the sum is over the (finitely many) primes $w$ of $K_\infty$ above $\ell$.
\sk

In view of Theorem~\ref{thm:mu-alg}, we write $\mu^\mathrm{alg}(\bar\rho)=0$ whenever
any of the $\mu$-invariants appearing in that result vanish. In that case, for any newform $f$
in $\mathcal{H}^-(\bar\rho)$ we may consider the $\lambda$-invariants
$\lambda({\rm Sel}(K_\infty,f))=\lambda(\mathfrak{Sel}(K_\infty,f))$.

\begin{theorem}\label{thm:lambda-alg}
Let $\bar\rho$ and $\Sigma$ be as above, and assume that $\mu^{\rm alg}(\bar\rho)=0$. If $f_1$ and $f_2$ are
any two newforms in 
$\mathcal{H}^-(\bar\rho)$ lying in the branches
$\mathbb{T}(\mathfrak{a}_1)$ and $\mathbb{T}(\mathfrak{a}_2)$, respectively,
then
\[
\lambda({\rm Sel}(K_\infty,f_1))-\lambda({\rm Sel}(K_\infty,f_2))=
\sum_{\ell\neq p}\delta_\ell(\mathfrak{a}_1)-\delta_\ell(\mathfrak{a}_2).
\]
\end{theorem}

\begin{proof}
Since we have the divisibilities $N^-\mid N(\mathfrak{a}_i)\mid N(\Sigma)$ with the quotient $N(\Sigma)/N^-$ only divisible
by primes that are split in $K$, the arguments of \cite[\S{4}]{EPW} apply verbatim (\emph{cf.} \cite[Thm.~7.1]{pollack-weston}).
\end{proof}



\section{Applications to the main conjecture}\label{sec:applications}

\subsection{Variation of anticyclotomic Iwasawa invariants}

Recall the definition of the analytic invariant
$e_\ell(\mathfrak{a})=\lambda(E_\ell(\mathfrak{a}))$,
where $E_\ell(\mathfrak{a})$ 
is the Euler factor from Section~\ref{subsec:comparison}, and of the algebraic invariant
$\delta_\ell(\mathfrak{\mathfrak{a}})$ introduced in $(\ref{def:delta-ell})$.


\begin{lemma}\label{5.1.5}
Let $\mathfrak a_1$, $\mathfrak a_2$ be minimal primes of $\mathbb T_\Sigma$.
For any prime $\ell\neq p$ split in $K$, we have
\[
\delta_\ell(\mathfrak a_1)-\delta_\ell(\mathfrak a_2)=e_\ell(\mathfrak a_2)-e_\ell(\mathfrak a_1).
\]
\end{lemma}

\begin{proof}
Let $\mathfrak{a}$ be a minimal prime of $\mathbb T_\Sigma$, let
$f$ be a newform in the branch $\mathbb{T}(\mathfrak{a})$, and let
$\wp_f\subseteq\mathfrak{a}$ be the corresponding height one prime.
Since $\ell=\mathfrak{l}\bar{\mathfrak{l}}$ splits in $K$, we have
\[
\bigoplus_{w\vert\ell}H^1(K_{\infty,w},A_f)=
\Biggl(\bigoplus_{w\vert\mathfrak{l}}H^1(K_{\infty,w},A_f)\Biggr)
\oplus\Biggl(\bigoplus_{w\vert\bar{\mathfrak{l}}}H^1(K_{\infty,w},A_f)\Biggr)
\]
and \cite[Prop.~2.4]{GV} immediately implies that
\[
Ch_{\Lambda}\Biggl(\bigoplus_{w\vert\ell}H^1(K_{\infty,w},A_f)^\vee\Biggr)
=E_\ell(f,\ell^{-1}\gamma_{\mathfrak{l}})\cdot E_\ell(f,\ell^{-1}\gamma_{\bar{\mathfrak{l}}}),
\]
where $E_\ell(f,\ell^{-1}\gamma_{\mathfrak{l}})\cdot E_\ell(f,\ell^{-1}\gamma_{\bar{\mathfrak{l}}})$
is the specialization of $E_\ell(\mathfrak{a})$ at $\wp_f$. 
The result thus follows from \cite[Lemma 5.1.5]{EPW}.
\end{proof}

\begin{theorem}\label{thm:variation}
Suppose that $\bar\rho$ satisfies Assumption~\SU. If for some newform $f_0\in\mathcal{H}^-(\bar\rho)$ we have
the equalities
\[
\mu({\rm Sel}(K_\infty,f_0))=\mu(L_p(f_0/K))=0\quad\textrm{and}\quad
\lambda({\rm Sel}(K_\infty,f_0))=\lambda(L_p(f_0/K)),
\]
then the equalities
\[
\mu({\rm Sel}(K_\infty,f))=\mu(L_p(f/K))=0\quad\textrm{and}\quad
\lambda({\rm Sel}(K_\infty,f))=\lambda(L_p(f/K))
\]
hold for all newforms $f\in\mathcal{H}^-(\bar\rho)$.
\end{theorem}

\begin{proof}
Let $f$ be any newform in $\mathcal{H}^-(\bar\rho)$. Since the algebraic and analytic
$\mu$-invariants of $f_0$ both vanish, the vanishing of $\mu({\rm Sel}(K_\infty,f))$ and $\mu(L_p(f/K))$
follows from Theorems~\ref{thm:mu-alg} and ~\ref{3.7.5}, respectively.
On the other hand, combining Theorems~\ref{3.7.7} and \ref{thm:lambda-alg}, and
Lemma~\ref{5.1.5}, we see that
\[
\lambda({\rm Sel}(K_\infty,f))-\lambda({\rm Sel}(K_\infty,f_0))=
\lambda(L_p(f/K))-\lambda(L_p(f_0/K)),
\]
and hence the equality $\lambda({\rm Sel}(K_\infty,f_0))=\lambda(L_p(f_0/K))$
implies the same equality for $f$.
\end{proof}

\subsection{Applications to the main conjecture}

As an immediate consequence of Weierstrass preparation theorem,
Theorem~\ref{thm:variation} together with one the divisibilities predicted by the
anticyclotomic main conjecture implies the full anticyclotomic main conjecture.


\begin{theorem}[Skinner--Urban]\label{thm:SU}
Let $f\in S_k(\Gamma_0(N))$ be a newform of weight $k\equiv 2\pmod{p-1}$ and trivial nebentypus.
Suppose that $\bar{\rho}_f$ satisfies Assumption~\SU~ and that $p$ splits in $K$. Then
\[
(L_p(f/K))\supseteq Ch_{\Lambda}({\rm Sel}(K_\infty,f)^\vee).
\]
\end{theorem}

\begin{proof}
This follows from specializing the divisibility in \cite[Thm.~3.26]{SU} to the anticyclotomic line.
Indeed, let $\mathbf{f}=\sum_{n\geqslant 1}\mathbf{a}_n(\mathbf{f})q^n\in\mathbb{I}\pwseries{q}$
be the $\Lambda$-adic form with coefficients in $\mathbb{I}:=\mathbb{T}(\mathfrak{a})^\circ$
associated with the branch of the Hida family containing 
$f$, 
let $\Sigma$ be a finite set of primes as in Section~\ref{subsec:2var-branch}, let $\Sigma'\supseteq\Sigma$
be a finite set of primes of $K$ containing $\Sigma$ and all primes dividing $pN(\mathfrak{a})D_K$,
and assume that $\Sigma'$ contains at least one prime $\ell\neq p$ that splits in $K$.
Under these assumptions, in \cite[Thm.~3.26]{SU} it is shown that
\begin{equation}\label{eq:SU}
(\mathfrak{L}_p^{\Sigma'}(\mathbf{f}/K))\supseteq
Ch_{\Lambda_{\mathbf{f}}(L_\infty)}(\mathfrak{Sel}^{\Sigma'}(L_\infty,A_{\mathbf{f}})^\vee),
\end{equation}
where $L_\infty=K_\infty K_{\rm cyc}$ is the $\Z_p^2$-extension of $K$, 
$\Lambda_{\mathbf{f}}(L_\infty)$ is the three-variable Iwasawa algebra
$\mathbb{I}\pwseries{{\rm Gal}(L_\infty/K)}$, and $\mathfrak{L}_p^{\Sigma'}(\mathbf{f}/K)$
and $\mathfrak{Sel}^{\Sigma'}(L_\infty,A_{\mathbf{f}})$ are the ``$\Sigma'$-primitive'' $p$-adic $L$-function and Selmer group
defined in \cite[\S{3.4.5}]{SU} and \cite[\S\S{3.1.3}, 3.1.10]{SU}, respectively.

Recall the character $\Theta:G_\Q\rightarrow\Z_p[[1+p\Z_p]]^\times$ from Section~\ref{subsec:crit}, regarded
as a character on ${\rm Gal}(L_\infty/K)$, and let
\[
{\rm Tw}_{\Theta^{-1}}:\Lambda_{\mathbf{f}}(L_\infty)\longrightarrow\Lambda_{\mathbf{f}}(L_\infty)
\]
be the $\mathbb{I}$-linear isomorphism induced by ${\rm Tw}_{\Theta^{-1}}(g)=\Theta^{-1}(g)g$ for
$g\in{\rm Gal}(L_\infty/K)$. Choose a topological generator $\gamma\in{\rm Gal}(K_{\rm cyc}/K)$, and expand
\[
{\rm Tw}_{\Theta^{-1}}(\mathfrak{L}_p^{\Sigma'}(\mathbf{f}/K))=
\mathfrak{L}_{p,0}^{\Sigma'}(\mathbf{f}/K)+\mathfrak{L}_{p,1}^{\Sigma'}(\mathbf{f}/K)(\gamma-1)+\cdots
\]
with $\mathfrak{L}_{p,i}^{\Sigma'}(\mathbf{f}/K)\in\Lambda_{\mathbf{f}}(K_\infty)=\mathbb{I}[[\Gamma]]$.
In particular, note that $\mathfrak{L}_{p,0}^{\Sigma'}(\mathbf{f}/K)$ is the restriction of the twisted
three-variable $p$-adic $L$-function ${\rm Tw}_{\Theta^{-1}}(\mathfrak{L}_p^{\Sigma'}(\mathbf{f}/K))$ to the ``self-dual'' plane.

Because of our assumptions on $f$, the $\Lambda$-adic form $\mathbf{f}$ has trivial tame character, and
hence denoting by ${\rm Frob}_\ell$ an arithmetic Frobenius at any prime $\ell\nmid N(\mathfrak{a})p$,
the Galois representation
\[
\rho(\mathfrak{a}):G_\Q\longrightarrow{\rm GL}(T_{\mathbf{f}})\simeq{\rm GL}_2(\mathbb{T}(\mathfrak{a})^\circ)
\]
considered in $\S\ref{sec:branches}$ (which is easily seen to agree with the twisted representation
considered in \cite[p.37]{SU}) is such that
\[
{\rm det}(X-{\rm Frob}_\ell\vert T_{\mathbf{f}})=X^2-\mathbf{a}_\ell(\mathbf{f})X+\Theta^2(\ell)\ell.
\]
The twist $T_{\mathbf{f}}^\dagger:=T_{\mathbf{f}}\otimes\Theta^{-1}$ is therefore self-dual. Thus combining
\cite[Lemma~6.1.2]{Rubin-ES} with a straightforward variant of \cite[Prop.~3.9]{SU}
having ${\rm Gal}(K_\infty/K)$ in place of ${\rm Gal}(K_{\rm cyc}/K)$, we see that divisibility $(\ref{eq:SU})$
implies that
\begin{equation}\label{eq:SU-}
(\mathfrak{L}_{p,0}^{\Sigma'}(\mathbf{f}/K))\supseteq
Ch_{\Lambda_{\mathbf{f}}(K_\infty)}(\mathfrak{Sel}^{\Sigma'}(K_\infty,A_{\mathbf{f}}^\dagger)^\vee).
\end{equation}
(Here, as above, $A_{\mathbf{f}}$ denotes the Pontryagin dual
$T_{\mathbf{f}}\otimes_{\mathbb{I}}{\rm Hom}_{\rm cts}(\mathbb{I},\Q_p/\Z_p)$, and $A_{\mathbf{f}}^\dagger$
is the corresponding twist.) We next claim that, setting $\Sigma'':=\Sigma'\smallsetminus\Sigma$, we have
\begin{equation}\label{eq:comp}
(\mathfrak{L}_{p,0}^{\Sigma'}(\mathbf{f}/K))=(L_\Sigma(\bar\rho,\mathfrak{a})\cdot
\prod_{\substack{v\in\Sigma''\\v\nmid p}}E_v(\mathfrak{a})),
\end{equation}
where $L_\Sigma(\bar\rho,\mathfrak{a})$ is the two-variable $p$-adic $L$-function
constructed in Section~\ref{subsec:2varL}, and if $v$ lies over the rational prime $\ell$,
$E_v(\mathfrak{a})$ is the Euler factor given by
\[
E_v(\mathfrak{a})=\det({\rm Id}-{\rm Frob}_vX\vert(V_{\mathbf{f}}^\dagger)_{I_v})_{X=\ell^{-1}{\rm Frob}_v},
\]
where $V_{\mathbf{f}}:=T_{\mathbf{f}}\otimes_{\mathbb{I}}{\rm Frac}(\mathbb{I})$,
and ${\rm Frob}_v$ is an arithmetic Frobenius at $v$.
(Note that for $\ell=\mathfrak{l}\bar{\mathfrak{l}}$ split in $K$,
$E_{\mathfrak{l}}(\mathfrak{a})\cdot E_{\bar{\mathfrak{l}}}(\mathfrak{a})$ is simply the Euler factor (\ref{def:e-a}).)
Indeed, combined with Theorem~\ref{thm:3.6.2}
and Theorem~\ref{thm:3.4.3}, equality $(\ref{eq:comp})$ specialized to any arithmetic prime
$\wp\subseteq\mathbb{T}(\mathfrak{a})$ of weight $2$ is shown in \cite[(12.3)]{SU}, from where the claim
follows easily from the density of these primes. (See also \cite[Thm.~6.8]{pollack-weston} for the
comparison between the different periods involved in the two constructions, which differ by a $p$-adic unit under our assumptions.)

Finally, $(\ref{eq:SU-})$ and $(\ref{eq:comp})$ combined with Theorem~\ref{thm:3.6.2} and \cite[Props.~2.3,8]{GV}
imply that
\[
(L(\bar\rho,\mathfrak{a}))\supseteq
Ch_{\Lambda_{\mathbf{f}}(K_\infty)}(\mathfrak{Sel}(K_\infty,A^\dagger_{\mathbf{f}})^\vee),
\]
from where the result follows by specializing at $\wp_f$ using Theorem~\ref{thm:3.4.3}
and Theorem~\ref{thm:mu-alg}.
\end{proof}

In the opposite direction, we have the following result:

\begin{theorem}[Bertolini--Darmon]\label{thm:BD}
Let $f=\sum_{n=1}^\infty a_n(f)q^n$ be a $p$-ordinary newform of weight $2$, level $N$, and trivial nebentypus.
Suppose that $\bar{\rho}_f$ satisfies Assumption~(CR) and that
\begin{equation}\label{eq:PO}
a_p(f)\not\equiv\pm{1}\pmod{p}.\tag{PO}
\end{equation}
Then
\[
(L_p(f/K))\subseteq Ch_{\Lambda}({\rm Sel}(K_\infty,f)^\vee).
\]
\end{theorem}

\begin{proof}
This is the main result of \cite{bdIMC}, as extended by Pollack--Weston \cite{pollack-weston} to
newforms of weight $2$ not necessarily defined over $\Q$ and under the stated hypotheses (weaker
that in \cite{bdIMC}) on $\bar\rho_f$. See also \cite{KPW} for a detailed discussion on the additional
``non-anomalous'' hypothesis (PO) on $f$.
\end{proof}


Before we combine the previous two theorems with our main results in this paper, we note that
condition~(PO) in Theorem~\ref{thm:BD} can be phrased in terms of the Galois representation $\rho_f$
associated to $f$. Indeed, let $f=\sum_{n=1}^\infty a_n(f)q^n$ be a $p$-ordinary newform as above,
defined over a finite extension $F/\Q_p$ with ring of integers $\cO$. Then we have
\[
\rho_f\vert_{D_p}\simeq\left(\begin{smallmatrix}\varepsilon&*\\0&\delta\end{smallmatrix}\right)
\]
on a decomposition group $D_p\subseteq G_\Q$ at $p$, with $\delta:D_p\rightarrow\cO^\times$ an unramified
character sending ${\rm Frob}_p$ to the unit root $\alpha_p$ of $X^2-a_p(f)X+p$. Since clearly
$\alpha\equiv a_p(f)\pmod{p}$, we see that condition~(PO) amounts to the requirement that
\begin{equation}\label{eq:PO2}
\delta({\rm Frob}_p)\not\equiv\pm{1}\pmod{p}.\tag{PO}
\end{equation}

Now we are finally in a position to prove our main application to the anticyclotomic Iwasawa main conjecture
for $p$-ordinary newforms.

\begin{corollary}\label{coro:5.4}
Suppose that $\bar\rho$ satisfies Assumptions~\SU~and (PO) and that $p$ splits in $K$, and let $f$ be a newform
in $\mathcal{H}^-(\bar\rho)$ of weight $k\equiv 2\pmod{p-1}$ and trivial nebentypus. Then
the anticyclotomic Iwasawa main conjecture holds for $f$.
\end{corollary}

\begin{proof}
After Theorem~\ref{thm:variation} and Theorem~\ref{thm:SU}, to check the anticyclotomic main conjecture for
\emph{any} newform $f$ as in the statement, it suffices to check the three equalities
\begin{equation}\label{Iw}
\mu({\rm Sel}(K_\infty,f_0))=\mu(L_p(f_0/K))=0\quad\textrm{and}\quad
\lambda({\rm Sel}(K_\infty,f_0))=\lambda(L_p(f_0/K)).
\end{equation}
holds \emph{some} $f_0\in\mathcal{H}^-(\bar\rho)$ of weight $k\equiv 2\pmod{p-1}$ and trivial nebentypus.

Let $\mathbb{T}(\mathfrak{a})$ be the irreducible component of $\mathcal{H}^-(\bar\rho)$ containing $f$,
and let $f_0\in S_2(\Gamma_0(Np))$ be the $p$-stabilized newform corresponding to an arithmetic prime $\wp\subseteq\mathbb{T}(\mathfrak{a})$
of weight $2$ and trivial nebentypus. By Assumption~(PO), the form $f_0$ is necessarily
the $p$-stabilization of a $p$-ordinary newform $f_0^\sharp\in S_2(\Gamma_0(N))$
(see e.g. \cite[Lemma~2.1.5]{howard-invmath}). From the combination of Theorem~\ref{thm:SU} and Theorem~\ref{thm:BD},
the anticyclotomic Iwasawa main conjecture holds for $f_0^\sharp$, and since we clearly have
\[
L_p(f_0/K)=L_p(f_0^\sharp/K)\quad\textrm{and}\quad
{\rm Sel}(K_\infty,f_0)\simeq{\rm Sel}(K_\infty,f_0^\sharp)
\]
(note that the latter isomorphism relies on the absolute irreducibility of $\bar\rho$),
the anticyclotomic Iwasawa main conjecture holds for $f_0$ as well. In particular, equalities $(\ref{Iw})$ hold for this $f_0$,
and the result follows.
\end{proof}


%


\bibliographystyle{amsalpha}
\bibliography{paper}
\end{document}